\documentclass[10pt,journal,twocolumn,twoside]{IEEEtran}

\usepackage{enumerate}
\usepackage{amsmath,amsthm,amssymb}
\usepackage{framed}
\usepackage{graphicx}
\usepackage{epstopdf}
\usepackage{caption}
\usepackage{cite}
\usepackage[lined,algonl,boxed]{algorithm2e}
\newcommand{\R}{\mathbb{R}}
\newcommand{\N}{\mathbb{N}}
\newcommand{\Blambda}{\boldsymbol\lambda}
\newcommand{\Bo}[1]{\boldsymbol{#1}}
\newcommand{\tableANDalgSize}{\small}

\newcommand{\bec}[1]{\bar{\vec{#1}}}
\newcommand{\hec}[1]{\hat{\vec{#1}}}

\newcommand{\BlueText}[1]{#1}

\usepackage{subfigure}

\usepackage{color}
\usepackage{xcolor}

\makeatletter
\renewcommand{\@algocf@capt@plain}{above}
\makeatother

\begin{document}

\newtheorem{defin}{Definition}
\newtheorem{theorem}{Theorem}
\newtheorem{prop}{Proposition}
\newtheorem{lemma}{Lemma}
\newtheorem{corollary}{Corollary}
\newtheorem{alg}{Algorithm}
\newtheorem{remark}{Remark}
\newtheorem{notations}{Notations}
\newtheorem{assumption}{Assumption}
\newtheorem{example}{Example}

\newcommand{\be}{\begin{equation}}
\newcommand{\ee}{\end{equation}}
\newcommand{\ba}{\begin{array}}
\newcommand{\ea}{\end{array}}
\newcommand{\bea}{\begin{eqnarray}}
\newcommand{\eea}{\end{eqnarray}}
\newcommand{\combin}[2]{\ensuremath{ \left( \ba{c} #1 \\ #2 \ea \right) }}
\newcommand{\diag}{{\mbox{diag}}}
\newcommand{\rank}{{\mbox{rank}}}
\newcommand{\dom}{{\mbox{dom{\color{white!100!black}.}}}}
\newcommand{\range}{{\mbox{range{\color{white!100!black}.}}}}
\newcommand{\image}{{\mbox{image{\color{white!100!black}.}}}}
\newcommand{\herm}{^{\mbox{\scriptsize H}}}  
\newcommand{\sherm}{^{\mbox{\tiny H}}}       
\newcommand{\tran}{^{\mbox{\scriptsize T}}}  
\newcommand{\tranIn}{^{\mbox{-\scriptsize T}}}  
\newcommand{\card}{{\mbox{\textbf{card}}}}
\newcommand{\asign}{{\mbox{$\colon\hspace{-2mm}=\hspace{1mm}$}}}
\newcommand{\ssum}[1]{\mathop{ \textstyle{\sum}}_{#1}}

\newcommand{\vbar}{\raisebox{.17ex}{\rule{.04em}{1.35ex}}}
\newcommand{\vbarind}{\raisebox{.01ex}{\rule{.04em}{1.1ex}}}
\newcommand{\D}{\ifmmode {\rm I}\hspace{-.2em}{\rm D} \else ${\rm I}\hspace{-.2em}{\rm D}$ \fi}
\newcommand{\T}{\ifmmode {\rm I}\hspace{-.2em}{\rm T} \else ${\rm I}\hspace{-.2em}{\rm T}$ \fi}
\newcommand{\B}{\ifmmode {\rm I}\hspace{-.2em}{\rm B} \else \mbox{${\rm I}\hspace{-.2em}{\rm B}$} \fi}
\newcommand{\Hil}{\ifmmode {\rm I}\hspace{-.2em}{\rm H} \else \mbox{${\rm I}\hspace{-.2em}{\rm H}$} \fi}
\newcommand{\C}{\ifmmode \hspace{.2em}\vbar\hspace{-.31em}{\rm C} \else \mbox{$\hspace{.2em}\vbar\hspace{-.31em}{\rm C}$} \fi}
\newcommand{\Cind}{\ifmmode \hspace{.2em}\vbarind\hspace{-.25em}{\rm C} \else \mbox{$\hspace{.2em}\vbarind\hspace{-.25em}{\rm C}$} \fi}
\newcommand{\Q}{\ifmmode \hspace{.2em}\vbar\hspace{-.31em}{\rm Q} \else \mbox{$\hspace{.2em}\vbar\hspace{-.31em}{\rm Q}$} \fi}
\newcommand{\Z}{\ifmmode {\rm Z}\hspace{-.28em}{\rm Z} \else ${\rm Z}\hspace{-.38em}{\rm Z}$ \fi}

\newcommand{\sgn}{\mbox {sgn}}
\newcommand{\var}{\mbox {var}}
\newcommand{\E}{\mbox {E}}
\newcommand{\cov}{\mbox {cov}}
\renewcommand{\Re}{\mbox {Re}}
\renewcommand{\Im}{\mbox {Im}}
\newcommand{\cum}{\mbox {cum}}

\renewcommand{\vec}[1]{{\bf{#1}}}     
\newcommand{\vecsc}[1]{\mbox {\boldmath \scriptsize $#1$}}     
\newcommand{\itvec}[1]{\mbox {\boldmath $#1$}}
\newcommand{\itvecsc}[1]{\mbox {\boldmath $\scriptstyle #1$}}
\newcommand{\gvec}[1]{\mbox{\boldmath $#1$}}

\newcommand{\balpha}{\mbox {\boldmath $\alpha$}}
\newcommand{\bbeta}{\mbox {\boldmath $\beta$}}
\newcommand{\bgamma}{\mbox {\boldmath $\gamma$}}
\newcommand{\bdelta}{\mbox {\boldmath $\delta$}}
\newcommand{\bepsilon}{\mbox {\boldmath $\epsilon$}}
\newcommand{\bvarepsilon}{\mbox {\boldmath $\varepsilon$}}
\newcommand{\bzeta}{\mbox {\boldmath $\zeta$}}
\newcommand{\boldeta}{\mbox {\boldmath $\eta$}}
\newcommand{\btheta}{\mbox {\boldmath $\theta$}}
\newcommand{\bvartheta}{\mbox {\boldmath $\vartheta$}}
\newcommand{\biota}{\mbox {\boldmath $\iota$}}
\newcommand{\blambda}{\mbox {\boldmath $\lambda$}}
\newcommand{\bmu}{\mbox {\boldmath $\mu$}}
\newcommand{\bnu}{\mbox {\boldmath $\nu$}}
\newcommand{\bxi}{\mbox {\boldmath $\xi$}}
\newcommand{\bpi}{\mbox {\boldmath $\pi$}}
\newcommand{\bvarpi}{\mbox {\boldmath $\varpi$}}
\newcommand{\brho}{\mbox {\boldmath $\rho$}}
\newcommand{\bvarrho}{\mbox {\boldmath $\varrho$}}
\newcommand{\bsigma}{\mbox {\boldmath $\sigma$}}
\newcommand{\bvarsigma}{\mbox {\boldmath $\varsigma$}}
\newcommand{\btau}{\mbox {\boldmath $\tau$}}
\newcommand{\bupsilon}{\mbox {\boldmath $\upsilon$}}
\newcommand{\bphi}{\mbox {\boldmath $\phi$}}
\newcommand{\bvarphi}{\mbox {\boldmath $\varphi$}}
\newcommand{\bchi}{\mbox {\boldmath $\chi$}}
\newcommand{\bpsi}{\mbox {\boldmath $\psi$}}
\newcommand{\bomega}{\mbox {\boldmath $\omega$}}

\newcommand{\bolda}{\mbox {\boldmath $a$}}
\newcommand{\bb}{\mbox {\boldmath $b$}}
\newcommand{\bc}{\mbox {\boldmath $c$}}
\newcommand{\bd}{\mbox {\boldmath $d$}}
\newcommand{\bolde}{\mbox {\boldmath $e$}}
\newcommand{\boldf}{\mbox {\boldmath $f$}}
\newcommand{\bg}{\mbox {\boldmath $g$}}
\newcommand{\bh}{\mbox {\boldmath $h$}}
\newcommand{\bp}{\mbox {\boldmath $p$}}
\newcommand{\bq}{\mbox {\boldmath $q$}}
\newcommand{\br}{\mbox {\boldmath $r$}}
\newcommand{\bs}{\mbox {\boldmath $s$}}
\newcommand{\bt}{\mbox {\boldmath $t$}}
\newcommand{\bu}{\mbox {\boldmath $u$}}
\newcommand{\bv}{\mbox {\boldmath $v$}}
\newcommand{\bw}{\mbox {\boldmath $w$}}
\newcommand{\bx}{\mbox {\boldmath $x$}}
\newcommand{\by}{\mbox {\boldmath $y$}}
\newcommand{\bz}{\mbox {\boldmath $z$}}



\title{On the Convergence of Alternating Direction Lagrangian Methods for Nonconvex Structured Optimization Problems }
\author{Sindri Magn\'{u}sson,
        Pradeep Chathuranga Weeraddana,~\IEEEmembership{Member,~IEEE,}
        \\Michael G. Rabbat,~\IEEEmembership{Member,~IEEE,}
        Carlo Fischione,~\IEEEmembership{Member,~IEEE,}
\thanks{S. Magn\'{u}sson, P. C. Weeraddana, and C. Fischione are with the Electrical Engineering School, Access Linnaeus Center, KTH Royal Institute of Technology, Stockholm, Sweden. 
 \textit{ \{sindrim, chatw, carlofi\}@kth.se}.  

 M. G. Rabbat is with the Department of Electrical and Computer Engineering, McGill University, Montr\'{e}al, Canada.  E-mail: michael.rabbat@mcgill.ca.
}}

\maketitle

\begin{abstract}
  Nonconvex and structured optimization problems arise in many engineering applications that demand scalable and distributed solution methods.
 The study of the convergence properties of these methods is in general difficult due to the nonconvexity of the problem.
 In this paper, two distributed solution methods that combine the fast convergence properties of augmented Lagrangian-based methods with the separability properties of  alternating optimization are investigated. 
 The first method is adapted from the classic quadratic penalty function method and is called the \emph{Alternating Direction Penalty Method} (ADPM). 
 Unlike the original quadratic penalty function method, in which single-step optimizations are adopted, ADPM uses an alternating optimization, which in turn makes it scalable. 
 The second method is the well-known \emph{Alternating Direction Method of Multipliers} (ADMM). 
 It is shown that ADPM for nonconvex problems asymptotically converges to a primal feasible point under mild conditions \BlueText{and an additional condition ensuring that it asymptotically reaches the standard first order necessary conditions for local optimality are introduced. }
 In the case of the ADMM, novel sufficient conditions under which the algorithm asymptotically reaches the standard first order necessary conditions are established.  
 \BlueText{ Based on this, complete convergence of ADMM for a class of low dimensional problems are characterized. }
 Finally, the results are illustrated by applying ADPM and ADMM to a nonconvex localization problem in wireless sensor~networks. 
\end{abstract}

\begin{keywords}\vspace{-0mm}
 Nonconvex Optimization, ADMM, Localization, Distributed Optimization
\end{keywords}

\section{Introduction}

\IEEEPARstart{T}{he} last few decades'  increasingly rapid technological developments have resulted in vast amounts of dispersed data. Optimization techniques have played a central role in transforming the vast data sets into usable information. However, due to the increasing size of the related optimization problems, it is essential that these optimization techniques scale with data size. 
 \BlueText{
 Fortunately, many large scale optimization problems in real world applications possess appealing structural properties, due to the networked nature of the problems.
 } 
 Thus, increasing research efforts have been devoted to the investigation of how these  structural properties can be exploited in the algorithm design to achieve scalability. 
 The focal point of these efforts has been on ``well-behaved" convex problems, rather than more challenging nonconvex problems.
  Nevertheless, large scale nonconvex problems arise in many real world network applications. 
  Examples of such nonconvex applications include matrix factorization techniques for recommender systems (the Netflix challenge)~\cite{Koren_2009}, localization in wireless sensor networks~\cite{Patwari_2005}, optimal power flow in smart grids~\cite{Frank_2012_1,Frank_2012_2},
 and LDPC decoding~\cite{Xishuo_2012}. 
 Interestingly, these large scale nonconvex applications tend to have the structural advantages that are commonly exploited to design scalable algorithms for their convex counterparts. 
This suggests that the algorithms used for large scale convex problems can potentially be applied to nonconvex problems as well.
  However, theoretical guarantees for these algorithms in the nonconvex regime have not yet been established.
\BlueText{
 This paper investigates convergence properties of a class of scalable and distributed algorithms for  \emph{nonconvex} structured optimization problems.
 Here, 
 (i) by \emph{distributed algorithms} we mean any algorithm that can be executed by at least two entities where no single entity has access to the full problem data, and 
 (ii) by \emph{structured optimization problems} we mean any problem with 
 structures in the problem data that can be exploited to achieve (i). 
}

\subsection{Related Literature} \label{Sec:relatedLiterature}

Many recent studies on large scale optimization have focused on distributed subgradient methods in the context of multi-agent networks~\cite{Nedic_2009,Nedic_2010,Zhu_2012,Duchi_2012,Jakovetic_2014,Shi_2014-Extra,Zhu_2013,Bianchi_2013}. There, multiple agents, each with a private objective function, cooperatively minimize the aggregate objective function by communicating over the network. 
 In contrast to~\cite{Nedic_2009,Nedic_2010,Zhu_2012,Duchi_2012,Jakovetic_2014,Shi_2014-Extra}, the papers~\cite{Zhu_2013} and \cite{Bianchi_2013} consider nonconvex multi-agent problems. 
 Specifically, \cite{Zhu_2013} applies distributed subgradient methods to the (convex) dual problem and investigates sufficient conditions under which the approach converges to a pair of optimal primal/dual variables. On the other hand, \cite{Bianchi_2013} studies the convergence of stochastic subgradient methods to a point satisfying the first order necessary conditions for local optimality with probability one. 
 \BlueText{
 A main drawback of these gradient based approaches is that they can only converge to an exact optimal (or local optimal) solution when a diminishing step size is used, which results in poor convergence rate. 
 The diminishing step size assumption is relaxed in the promising recent work~\cite{Shi_2014-Extra} while keeping the exact convergence by introducing a correction term, which significantly improves the convergence rate. 
}

Another widely used approach for structured convex optimization is the Alternating Direction Method of Multipliers (ADMM)~ \cite{Gabay1976,Eckstein1992,Boyd2011_ADMM}.  ADMM is a variant of the classical method of multipliers (MM)~\cite[Chapter 2]{constrainted_bertsekas} \cite[Chapter 4.2]{nonlinear_bertsekas}, where the primal variable update of the MM is split into subproblems, whenever the objective  is separable. This structure is common in large scale optimization problems that arise in practice~\cite{Boyd2011_ADMM}. Even problems that do not possess such a structure can often be posed equivalently in a form appropriate for ADMM by introducing auxiliary variables and linear constraints. These techniques have been employed in many recent works when designing distributed algorithms for convex, as well as nonconvex problems~~\cite{6684590,Mateos_2010,Schizas_2008,Boyd2011_ADMM,2013arXiv1307.8254W,Shi-Ling-Yuan-Wu-Yin-2014,Magnusson-Weeraddana-Fischione-2014,Iutzeler-2013-Explicit}. 
 A key property of ADMM compared with other existing scalable approaches, such as subgradient and dual descent methods (mentioned above) is its superior convergence behavior, see~\cite{Boyd2011_ADMM,Mateos_2010,Ghadimi-2013-Optimal} for empirical results. Characterizing the exact convergence rate of ADMM is still an ongoing research topic~\cite{Shi-Ling-Yuan-Wu-Yin-2014,Hong_2012,Iutzeler-2013-Explicit,Ghadimi-2013-Optimal}. Many recent papers have also numerically demonstrated the fast and appealing convergence behavior of ADMM even on nonconvex problems~\cite{Kanamori_2013,Dehua_2013,Liu_2012,Chartrand_2012,Magnusson-Weeraddana-Fischione-2014}. 
Despite these encouraging observations, there are still no theoretical guarantees for ADMM's  convergence in the nonconvex regime. 
Therefore, investigating convergence properties of the ADMM and related algorithms in nonconvex settings is of great importance in theory as well as in practice, and is motivated by the many emerging large scale nonconvex applications.

\subsection{Notation and Definitions}


 Vectors and matrices are represented by boldface lower and upper case letters, respectively.
 The set of real and natural numbers are denoted by $\R$ and $\N$, respectively. 
 The set of real $n$ vectors and $n{\times} m$ matrices are denoted by $\R^n$ and $\R^{n\times m}$, respectively.
 The $i$th component of the vector $\vec{x}$ is denoted by $\vec{x}_i$.
 The superscript $(\cdot)\tran$ stands for transpose.
 We use parentheses to construct vectors and matrices from comma separated lists as  $(\vec{x}_1,{\cdots},\vec{x}_n){=}[\vec{x}_1\tran, {\cdots}, \vec{x}_n\tran]\tran$ and $(\vec{A}_1,{\cdots},\vec{A}_n){=}[\vec{A}_1\tran, {\cdots}, \vec{A}_n\tran]\tran$, respectively. 
 $\diag(\vec{A}_1, {\cdots} , \vec{A}_n)$ denotes the  diagonal block matrix with $\vec{A}_1, {\cdots} , \vec{A}_n$ on the diagonal.
 $\vec{A}{\succ} 0$  ($\vec{A} {\succeq} 0$) indicates that the square matrix $\vec{A}$ positive (semi)definite.
 \BlueText{$||\cdot||$ denotes the $2$-norm.}
  We use the following definition.

    %

  \begin{defin}[\BlueText{\textbf{FON}}] \label{def:def_first_order_necessary_condition}
   Consider the optimization problem
\begin{equation} \label{eq:def_first_order_necessary_condition}
  \begin{aligned}
    & \underset{\vec{x}\in \R^p}{\text{ \emph{ minimize}}}
    & & f(\vec{x}) \\
    & \text{ \emph{ subject to }}
    & & \Bo\phi(\vec{x})=\vec{0}, \quad \Bo\psi(\vec{x})\leq \vec{0}
  \end{aligned}
\end{equation}
   where $\Bo\phi {:}\R^p \rightarrow \R^{q_1}$ and $\Bo\psi{:}\R^p \rightarrow \R^{q_2}$  are continuously differentiable functions.
    We say that $\vec{x}^{\star}\in \R^p$ and $(\Bo\lambda^{\star},\Bo\mu^{\star} ) \in \R^{q_1+q_2}$ satisfy the first order necessary \BlueText{(FON)}  conditions for problem~\eqref{eq:def_first_order_necessary_condition}, if following hold.
   1) Primal feasibility: $\Bo\phi(\vec{x}^{\star}){=}\vec{0}$ and $\Bo\psi(\vec{x}^{\star})\leq \vec{0}$.
   2) Dual feasibility:   $\Bo\mu^{\star} {\geq} \vec{0}$.
   3) Complementary slackness: $(\Bo\mu^{\star})_i\Bo\psi_i (\vec{x^{\star}}){=}0$, $i{=}1,\cdots q_2$.
   4) Lagrangian vanishes: $\nabla f(\vec{x}^{\star}) {=} \nabla \Bo\phi(\vec{x}^{\star}) \Bo\lambda^{\star} {+} \nabla \Bo\psi(\vec{x}^{\star}) \Bo\mu^{\star}$.
     We refer to $\vec{x}^{\star}$ and $(\Bo\lambda^{\star},\Bo\mu^{\star})$ as the primal and dual variables, respectively.
  \end{defin}

\section{Problem Statement, Related Background, and Contribution of the Paper} \label{sec:Problem_Statement_and_Background}

\BlueText{

 This section is organized as follows. 
 Section~\ref{sec:problemStatement} introduces the class of nonconvex structured problems we study.
 We give the necessary background on centralized algorithms in 
  Section~\ref{sec:Problem_Statement_and_Background_Augmented_Lagrangian_Methods},
  before introducing distributed algorithms which exploit the special structures of the related problems in Section~\ref{sec:Problem_Statement_and_Background_ADLM}.
 Then we state the contribution and organization of the paper in Section~\ref{sec:contributionOfThePaper}.


}

\subsection{Problem Statement} \label{sec:problemStatement}
 We consider the following optimization problem
\begin{equation} \label{eq:main_problem_formulation}
  \begin{aligned}
    & \underset{ \vec{x}\in \R^{p_1} ,\vec{z} \in \R^{p_2}}{\text{minimize}}
    & & f(\vec{x})+g(\vec{z}) \\
    & \text{subject to}
    & & \vec{x} \in \mathcal{X},~ \vec{z} \in \mathcal{Z} \\
    &&&  \vec{A}\vec{x}+\vec{B}\vec{z}=\vec{c},
  \end{aligned}
\end{equation}
 where $\vec{A}{\in}\R^{q\times p_1}$, $\vec{B}{\in} \R^{q\times p_2}$, and $\vec{c} {\in} \R^q$. 
  The use of the variable notation $\vec{x}$ and $\vec{z}$ is consistent with the literature~\cite{Boyd2011_ADMM}.
  The functions $f{:}\mathcal{X}{\rightarrow} \R$ and  $g{:}\mathcal{Z}{\rightarrow} \R$ are continuously differentiable on $\R^{p_1}$ and $\R^{p_2}$, respectively, and  may be \emph{nonconvex}. 
 We refer to the affine constraint $\vec{A}\vec{x}{+}\vec{B}\vec{z}{=}\vec{c}$ as \emph{the coupling constraint}. 
 We assume that Problem~\eqref{eq:main_problem_formulation} is feasible. 
 Problem~\eqref{eq:main_problem_formulation} is general in the sense that many interesting large scale problems, including consensus, and sharing~\cite[Section 7]{Boyd2011_ADMM}, among others can be equivalently posed in its form. 
\BlueText{
 Moreover, as noted in Section~\ref{Sec:relatedLiterature}, problem~\eqref{eq:main_problem_formulation} commonly appears in multi-agent networks, where $\vec{x}$ usually represents the private variable of each node/agent, $\vec{z}$ represents the coupling between the nodes, and the coupling constraint enforces the network consensus. }
 Therefore, our analytical results in subsequent sections apply to a broad class of problems of practical~importance.

\BlueText{
 Next we discus centralized solution methods  for Problem~\eqref{eq:main_problem_formulation} which are the basis for the distributed methods we study.

} 

\subsection{Penalty and Augmented Lagrangian Methods} \label{sec:Problem_Statement_and_Background_Augmented_Lagrangian_Methods}


  Nonconvex problems of the form~\eqref{eq:main_problem_formulation} can be gracefully handled by penalty and augmented Lagrangian methods, such as the quadratic penalty function method and method of multipliers, \cite[Chapter 2]{constrainted_bertsekas} \cite[Chapter 4.2]{nonlinear_bertsekas}.
  The main ingredient of these methods is the augmented Lagrangian, given by
\makeatletter%
\if@twocolumn%
  \begin{align*}
     L_{\rho}(\vec{x},\vec{z},\vec{y}) =& ~f(\vec{x})+g(\vec{z}) + \vec{y}\tran (\vec{A}\vec{x}+\vec{B}\vec{z}-\vec{c}) \\&+ (\rho/2) || \vec{A}\vec{x}+\vec{B}\vec{z}-\vec{c}||^2.
   \end{align*}
\else
\begin{equation*}
     L_{\rho}(\vec{x},\vec{z},\vec{y}) {=} f(\vec{x}){+}g(\vec{z}) {+} \vec{y}\tran (\vec{A}\vec{x}{+}\vec{B}\vec{z}{-}\vec{c}) {+} (\rho/2) || \vec{A}\vec{x}{+}\vec{B}\vec{z}{-}\vec{c}||^2.
\end{equation*}
\fi
\makeatother
 Here $\vec{x}$ and $\vec{z}$ are the primal variables of Problem~\eqref{eq:main_problem_formulation} and $\vec{y}\in \R^q$ and $\rho\in \R$ are referred to as the multiplier vector and the penalty parameter, respectively.

 The penalty and augmented Lagrangian methods consist in iteratively updating the variables $\vec{x}$, $\vec{z}$, $\vec{y}$, and $\rho$.
 An update common to all the methods is the primal variable update, i.e.,
\begin{equation} \label{eq:augmented_Lagrangian_formulation}
    (\vec{x}(t{+}1),\vec{z}(t{+}1)) = \underset{(\vec{x},\vec{z})\in \mathcal{X}\times \mathcal{Z}}{\text{argmin}}  L_{\rho(t)}(\vec{x},\vec{z},\vec{y}(t)),
\end{equation}
 where $t\in \N$ is the iteration index.
 The main difference between the two methods lies in the $\vec{y}$ and $\rho$ updates.
\BlueText{
 For example, in the case of the quadratic penalty method, the penalty parameter $\rho(t)$ is chosen such that $\lim_{t\rightarrow \infty} \rho(t)=\infty$ with the intention of enforcing the limit points of $\{ (\vec{x}(t),\vec{z}(t))\}_{t\in \N}$ to satisfy the coupling constraint.   
  It turns out that if the Lagrange multipliers are bounded, i.e., there exists $M\in \R$ such that $||\vec{y}(t)||<M$ for all $t\in \N$, %
 then every limit point of the sequence $\{ (\vec{x}(t),\vec{z}(t))\}_{t\in \N}$  is a global minimum of Problem~\eqref{eq:main_problem_formulation}~\cite[Proposition 2.1]{constrainted_bertsekas}.
}

\BlueText{

 The motive of the method of multipliers is to choose the sequence of multipliers $\{\vec{y}(t)\}_{t\in \N}$ intelligently to enable convergence to local or global optima of~\eqref{eq:main_problem_formulation} without needing $\lim_{t\rightarrow \infty} \rho(t){=}\infty$. 
 The well-known choice of $\{\vec{y}(t)\}_{t\in \N}$ in the method of multipliers follows 
  the recursion  
  \begin{equation} \vec{y}(t{+}1)=\vec{y}(t)+\rho(t)(\vec{x}(t{+}1)+\vec{z}(t{+}1)). \label{eq:Sec2B_dual-update} \end{equation}
  The motivation for~\eqref{eq:Sec2B_dual-update} is that when $(\vec{x}(t{+}1),\vec{z}(t{+}1))$ is locally optimal for Problem~\eqref{eq:augmented_Lagrangian_formulation} and satisfies the FON conditions (Definition~\ref{eq:def_first_order_necessary_condition})\footnote{We do not include the multipliers related to the constraint $\mathcal{X}{\times} \mathcal{Z}$ to simplify the presentation, but it is easily checked that the claim holds when they are included. } 
 then $(\vec{x}(t{+}1),\vec{z}(t{+}1))$ and $\vec{y}(t{+}1)$ satisfy conditions 2), 3), and 4) of the FON conditions for the original Problem~\eqref{eq:main_problem_formulation}, all except 1) primal feasibility. 
 Furthermore, under mild conditions, the method of multipliers converges to a local optimal point $(\vec{x}^{\star},\vec{z}^{\star})$ and to a corresponding optimal Lagrangian multiplier $\vec{y}^{\star}$~\cite[Proposition 2.4]{constrainted_bertsekas}.
 In addition to the local convergence, when $(\vec{x}(t),\vec{z}(t))$ is a global optima of~\eqref{eq:augmented_Lagrangian_formulation}, then~\eqref{eq:Sec2B_dual-update} is a gradient ascent step for the dual problem.
 However, due to non-zero duality gap in most nonconvex problems, the solution to~\eqref{eq:main_problem_formulation} can not be recovered from the dual problem. 
  Hence the method of multipliers can generally only be considered a local method.}



 In general, the penalty and augmented Lagrangian methods mentioned above are very reliable and effective for handling problems of the form~\eqref{eq:main_problem_formulation}.
 However, these methods entail centralized solvers, especially in the $(\vec{x},\vec{z})$-update~\eqref{eq:augmented_Lagrangian_formulation}, even if the objective function of problem~\eqref{eq:main_problem_formulation} has a desirable separable structure in $\vec{x}$ and $\vec{z}$.
 More specifically, these methods do not allow the possibility of performing the $(\vec{x},\vec{z})$-update in two steps: first $\vec{x}$-update and then $\vec{z}$-update.
 Otherwise, the assertions on the convergence of the algorithms do not hold anymore.
 Therefore, the penalty and augmented Lagrangian methods are not applicable in distributed settings, whenever the problems possess decomposition structures.
  Such restrictions have motivated an adaptation of the classical penalty and augmented Lagrangian methods that has excellent potential for a parallel/distributed implementation which we~discussed~now. 

\subsection{Alternating Direction Lagrangian Methods} \label{sec:Problem_Statement_and_Background_ADLM}

 Recall that problem~\eqref{eq:main_problem_formulation} has a linear coupling constraint and an objective function that is separable in $\vec{x}$ and $\vec{z}$.
 This motivates potential solution approaches to Problem~\eqref{eq:main_problem_formulation}, where the optimization in~\eqref{eq:augmented_Lagrangian_formulation} is performed in two steps, first in the $\vec{x}$ coordinate and then in the $\vec{z}$ coordinate, i.e.,
\begin{align}
      \vec{x}(t{+}1)&=\underset{\vec{x}\in \mathcal{X}}{\text{argmin}}~  L_{\rho(t)}(\vec{x},\vec{z}(t),\vec{y}(t)), \label{eq:ADLM-x-update}\\
      \vec{z}(t{+}1)&=\underset{\vec{z}\in \mathcal{Z}}{\text{argmin}}~L_{\rho(t)}(\vec{x}(t{+}1),\vec{z},\vec{y}(t)). \label{eq:ADLM-z-update}
\end{align}
 Let us refer to these approaches as \emph{Alternating Direction Lagrangian Methods} (ADLM).
 We consider two ADLM variants.
 The first variant is analogous to the quadratic penalty approach, where the sequence of penalty parameters $\{\rho(t)\}_{t\in \N}$ and the multiplier vectors $\{\vec{y}(t)\}_{t\in \N}$ are taken to be nondecreasing/divergent and bounded, respectively.
 \BlueText{We refer to this novel approach as the \emph{Alternating Direction Penalty Method} (ADPM).}
The second variant is the classic ADMM itself, the analog of the method of multipliers.
 We now pose the question: \emph{can the convergence of the considered ADLM variants, ADPM and ADMM, still be guaranteed when Problem~\eqref{eq:main_problem_formulation} is nonconvex?}


 \subsection{Contribution and Structure of the Paper} \label{sec:contributionOfThePaper}


\BlueText{
 We start by investigating the convergence behavior of the ADPM in Section~\ref{sec:ADPM} when Problem~\eqref{eq:main_problem_formulation} is nonconvex. 
 We consider a) an \emph{unconstrained} case in Section~\ref{sec:ADPM-analytic-props-unconstraint}, i.e., where $\mathcal{X}=\R^{p_1}$ and $\mathcal{Z}=\R^{p_2}$, and b) a \emph{constrained} case in Section~\ref{sec:ADPM-reaching-feasibility} where $\mathcal{X}$ and $\mathcal{Z}$ are compact sets.
 The analysis in case a) is based on assumptions on~\eqref{eq:main_problem_formulation} which highlight the situation when the $\vec{x}$- and $\vec{z}$- updates of ADLM are used to achieve distributed algorithms over networks and the coupling constraint expresses the network consensus.     
 Under these assumptions, we show that if $\vec{y}(t){=}\vec{0}$ and $\lim_{t\rightarrow }\rho(t)=\infty$, then the primal feasibility of~\eqref{eq:main_problem_formulation} is asymptotically achieved as ADPM proceeds. 
 In addition, if the sequence $1/\rho(t)$ is also non-summable and $(\vec{x}(t),\vec{z}(t))$ converge to $(\vec{x}^{\star},\vec{z}^{\star})$, then $(\vec{x}^{\star},\vec{z}^{\star})$ satisfies the FON conditions (Definition~\ref{eq:def_first_order_necessary_condition}) of~\eqref{eq:main_problem_formulation}. 
 In case b), we consider more general assumptions on~\eqref{eq:main_problem_formulation} and allow $\vec{y}(t)$ to be any bounded sequence.
 Under these assumptions, we show that if $\mathcal{X}$ and $\mathcal{Z}$ are convex and the sequence $1/\rho(t)$ is summable, then the primal feasibility of~\eqref{eq:main_problem_formulation} is asymptotically achieved as ADPM proceeds. 
 Moreover, we give an intuitive example showing why we need the sets $\mathcal{X}$ and $\mathcal{Z}$ to be convex in general.


 Next we investigate the convergence behavior of the ADMM  when~\eqref{eq:main_problem_formulation} is nonconvex in Section~\ref{sec:ADMM}. 
 We assume that the penalty parameter is fixed, i.e., $\rho(t)=\rho$. 
 We consider general assumptions on Problem~\eqref{eq:main_problem_formulation} where the sets $\mathcal{X}$ and $\mathcal{Z}$ can even be nonconvex. 
 We show that when $\vec{y}(t)$ converges then any limit point of $\vec{x}(t),\vec{z}(t))$ satisfies the FON conditions of Problem~\eqref{eq:main_problem_formulation}. 
 We note that the condition can be checked a posteriori or at runtime, by inspecting some algorithm parameters as the algorithm proceeds (online). 
 Moreover, we show how our results can be used to completely characterize the convergence of ADMM  for a class of problems, i.e., to determine to which point ADMM converges given an initialization. 
 In comparison to~\cite{Zhu_2013}, we consider ADMM, whereas therein  the standard Lagrangian dual function is  maximized.

 


 Finally, we illustrate how the considered methods can be applied to design distributed algorithms for cooperative localization in wireless sensor networks.} 



\section{Alternating Direction Penalty Method} \label{sec:ADPM}
  In this section we study convergence properties of the ADPM for addressing Problem~\eqref{eq:main_problem_formulation}.
  In Section~\ref{eq:ADPM-alg-descript} we give an explicit algorithm description and in Sections~\ref{sec:ADPM-analytic-props-unconstraint} and~\ref{sec:ADPM-reaching-feasibility} we investigate properties of the ADPM when $\mathcal{X}\times\mathcal{Z}=\R^{p_1}\times \R^{p_2}$ and when $\mathcal{X}\times\mathcal{Z} \subsetneqq \R^{p_1}\times \R^{p_2}$, respectively. 

{\tableANDalgSize

 \subsection{Algorithm Description} \label{eq:ADPM-alg-descript}
  The steps of ADPM are shown in Algorithm~1
\noindent\rule{\linewidth}{0.3mm}
\\
\emph{Algorithm 1: \ \textsc{The Alternating Direction Penalty Method (\small{ADPM}) }}

\vspace{-0.2cm}
\noindent\rule{\linewidth}{0.3mm}
\begin{enumerate}
   \item \label{Alg:ADPM-initialization-step}  \textbf{Initialization:} Set $t=0$ and initialize $\vec{z}(0)$, $\vec{y}(0)$, and $\rho(0)$.
   \item \label{Alg:ADPM-x-step} \textbf{\vec{x}-update:}  $\vec{x}(t+1){=}\underset{\vec{x}\in \mathcal{X}}{\text{argmin}}~  L_{\rho(t)}(\vec{x},\vec{z}(t),\vec{y}(t))$.
   \item \label{Alg:ADPM-z-step} \textbf{\vec{z}-update:} $\vec{z}(t+1){=}\underset{\vec{z}\in \mathcal{Z}}{\text{argmin}}~L_{\rho(t)}(\vec{x}(t+1),\vec{z},\vec{y}(t))$.
   \item \textbf{$\rho/\vec{y}$-update:} Update $\rho(t{+}1)$ and $\vec{y}(t{+}1)$.
   \item \label{Alg:ADPM-stop} \textbf{Stopping criterion:} If stopping criterion is met terminate, otherwise set $t=t+1$ and go to step 2.
\end{enumerate}
\vspace{-3mm}
\rule{\linewidth}{0.3mm}

}


%


\noindent
 The algorithm parameters $\rho(t)$ and $\vec{y}(t)$ are chosen such that 
 $\lim_{t\rightarrow \infty} \rho(t){=}\infty$ and the sequence $\{\vec{y}(t)\}_{t\in\N}$ is taken to be bounded.
 The $\vec{x}$- and $\vec{z}$- updates (steps \ref{Alg:ADPM-x-step} and \ref{Alg:ADPM-z-step}) are the main steps of the algorithm where the augmented Lagrangian is minimized in two steps.

  Nonconvexities of $f$ and $g$ suggest potential difficulties in the implementation of the $\vec{x}$- and $\vec{z}$- updates (see steps 2 and 3).
  However, it is worth noting that problems encountered in practice often contain structure that can be exploited to successfully implement the $\vec{x}$- and $\vec{z}$- updates.
  Several examples are given next.
 \begin{example} \label{example:subproblem-become-strongly-convex}

    Let $\mathcal{X}$ (or $\mathcal{Z}$) be convex, let $f$ (or $g$) be twice continuously differentiable, and suppose there exits $M\in \R$ such that $\nabla^2 f(\vec{x})>\alpha$ for all $\vec{x}\in \mathcal{X}$.
    Moreover, suppose $\vec{A}$ (or $\vec{B}$) has full column rank.
    Then the optimization problem in the $\vec{x}$-update (or $\vec{z}$-update) is strongly convex for sufficiently large $\rho(t)>-\alpha/\lambda_{\min}(\vec{A}\tran \vec{A})$.
  This can be seen by looking at the Hessian $\nabla_{\vec{x}}^2 L_{\rho(t)} (\vec{x},\vec{z}(t),\vec{y}(t))$ and using that $\vec{A}\tran \vec{A}$ is positive definite.
 \end{example}
 \begin{example}\label{example:subproblem-become-strongly-convex-2}
    Let $f(\vec{x})=\vec{x}\tran \vec{Q}\vec{x}+\vec{q}\tran \vec{x}$ where $\vec{Q} \in \R^{p_1\times p_1}$ is a symmetric indefinite  matrix.
    Then if $\vec{x}\tran \vec{Q} \vec{x}>0$ for all $\vec{x} \in \R^{p_1} \setminus  \{ \vec{0}\}$ in the null space of $\vec{A}$, then there exists $\bar{\rho}\in \R$ such that $L_{\rho(t)} (\cdot,\vec{z}(t),\vec{y}(t))$ is convex in $\vec{x}$ for all $\rho(t)\geq \bar{\rho}$, see\cite[Lemma 3.2.1 and Figure 3.2.1]{nonlinear_bertsekas}.
 \end{example}
 \begin{example} \label{example:subproblem-become-strongly-convex-3}
   A potential feature of the multi-agent setting is that the $\vec{x}$-  update is separable into low dimensional problems.
  More specifically, suppose the variable $\vec{x}$ is partitioned into low dimensional subvectors as $\vec{x}=( \vec{x}_1, \cdots , \vec{x}_N)$, where there is no coupling between $\vec{x}_i$ and $\vec{x}_j$ in the constraints, for all $i,j=1,\cdots,N$ such that $i\neq j$.
 Suppose also that the objective function is separable with respect to the partition, i.e., $f(\vec{x})=\sum_{i=1}^{N} f_i(\vec{x}_i)$.
  Then the objective function in the $\vec{x}$-update is also separable with respect to the partition.
  Thus, provided that each subvector $\vec{x}_i$ is of low dimension, global methods such as branch and bound can be efficiently used to optimally solve the optimization problem in the $\vec{x}$-update.
 \end{example}

\BlueText{
\subsection{ Algorithm Properties: Unconstrained Case} \label{sec:ADPM-analytic-props-unconstraint}

 In this section, we derive the convergence properties of the ADPM algorithm when $\mathcal{X}=\R^n$ and $\mathcal{Z}=\R^m$.
 Our convergence results assert that i) primal feasibility of problem~\eqref{eq:main_problem_formulation} is satisfied and ii) if the sequence $1/\rho(t)$ is non-summable and $(\vec{x}(t),\vec{z}(t))$ converges to a point $(\vec{x}^{\star},\vec{z}^{\star})$, then $(\vec{x}^{\star},\vec{z}^{\star})$ satisfies the FON conditions (Definition~\ref{eq:def_first_order_necessary_condition}) of Problem~\eqref{eq:main_problem_formulation}. 
 To establish this result precisely, let us first make the following assumptions.
\begin{assumption} \label{assumption:ADPM-unCon-1}
  $g(\vec{x}){=}0$, $\vec{A}{=}\vec{I}$,  $\vec{c}{=}\vec{0}$, $\vec{B}$ has full column rank. 
\end{assumption}
\begin{assumption} \label{assumption:ADPM-unCon-2}
 At least one of the following conditions holds true:
  \begin{enumerate}[a.]
    \item $f$ is continuously differentiable with bounded gradient, i.e., there exists $\kappa\in \R$ such that $|| \nabla f(\vec{x}) || {\leq} \kappa$ for all $\vec{x} \in \R^n$.
    \item 
      $||\vec{B}||_{\infty} {\leq} 1$ and  $||(\vec{B}\tran \vec{B})^{-1}\vec{B}\tran||_{\infty} {\leq} 1$. 
  Moreover, there  exist a scalar $c{>}0$ such that: (b.i) $[\nabla f(\vec{x})]_i < 0$ if $\vec{x}_i< {-}c$, for component $i \in \{ 1, {\cdots}, p_1\}$ and (b.ii) $[\nabla f(\vec{x})]_i > 0$ if $\vec{x}_i>c$, for $i \in \{1,\cdots, p_1\}$.   
  \end{enumerate}
\end{assumption}

 Assumption~\ref{assumption:ADPM-unCon-1} naturally arises when designing distributed algorithm over networks, where $\vec{x}$ represents private variables of each node/agent and $\vec{z}$ represents the coupling between the nodes.
   Assumption~\ref{assumption:ADPM-unCon-2}.a is standard in the literature, e.g., in relation to (sub)gradient methods methods~\cite{nonlinear_bertsekas,Nedic_2009,Jakovetic_2014}. 
 In addition, Assumption~\ref{assumption:ADPM-unCon-2}.b ensures that our results hold for more general classes of practical problems than covered by Assumption~\ref{assumption:ADPM-unCon-2}.a, e.g., when $f$ is a polynomial of even degree with positive leading coefficient (see Problem~\eqref{eq:distributed_localization_problem} in Section~\ref{sec:general applications}). 
 We note that the  $||\vec{B}||_{\infty} {\leq} 1$ and  $||(\vec{B}\tran \vec{B})^{-1}\vec{B}\tran||_{\infty} {\leq} 1$ naturally hold when
  $\vec{x}$ and $\vec{z}$ represent private and coupling variables of each node/agent in a connected network, e.g. see Section~\ref{sec:general applications}.
 The main implication of Assumption~\ref{assumption:ADPM-unCon-2}.b is that it ensures that the sequence $(\vec{x}(t),\vec{z}(t))$ is bounded as we show in the following lemma.
\begin{lemma}\label{sec3ADPM:Lemma1}
 Suppose Assumption~\ref{assumption:ADPM-unCon-2}.b holds true and $||\vec{z}(t)||_{\infty}\leq c$, then $||\vec{x}(t{+}1)||_{\infty}\leq c$ and $||\vec{z}(t{+}1)||_{\infty}{\leq} c $.
\end{lemma}
\begin{proof}
  Let us start by showing that $||\vec{x}(t{+}1)||_{\infty}\leq c$ by using contradiction. 
  Without loss of generality, we assume that $\vec{x}_i(t{+}1)<-c$ for some $i=1,\cdots, p_1$  (the other cases follow symmetrical arguments). 
  Then $[\nabla f(\vec{x}(t))]_i<0$, from Assumption~\ref{assumption:ADPM-unCon-2}.b, which in turn implies that 
 \begin{equation} ||(1/\rho) \nabla f(\vec{x}(t{+}1))+\vec{x}(t{+}1)||_{\infty}>c. \label{sec3a:lemma-10} \end{equation} 
  However, using the  FON conditions of the $\vec{x}$-update and  that $||\vec{B}||_{\infty}\leq1$ we also have
 \begin{align} 
   ||(1/\rho) \nabla f(\vec{x}(t{+}1)){+}\vec{x}(t{+}1)||_{\infty} &{=}||\vec{B} \vec{z}(t) ||_{\infty} \leq c \label{sec3a:lemma-1a}
 \end{align}
  Clearly, \eqref{sec3a:lemma-10} and~\eqref{sec3a:lemma-1a} contradict each other.
  Hence, $||\vec{x}(t{+}1)||_{\infty}{\leq} c$.


 
 Let us next show that $||\vec{z}(t{+}1)|| \leq c$.
  From the FON conditions of the $\vec{z}$-update we get that
   $\vec{z}(t{+}1)=(\vec{B}\tran \vec{B})^{- 1}\vec{B}\tran \vec{x}(t{+}1)$, which together with $||(\vec{B}\tran \vec{B})^{-1} \vec{B}\tran ||_{\infty}\leq 1$ ensures that $||\vec{z}(t{+}1)||\leq || \vec{x}(t{+}1)||_{\infty} \leq c$. 
\end{proof}

\noindent  We are now ready to derive the main result of this subsection. 

\begin{prop} \label{prop:ADPM-unCon-feas}

   Suppose assumptions~\ref{assumption:ADPM-unCon-1} and~\ref{assumption:ADPM-unCon-2} hold. 
   Let $r(t)$ be the residual at iteration $t$ of the ADPM defined as $r(t)=||\vec{x}(t)+\vec{B}\vec{z}(t)||$. 
   Then
 \begin{enumerate}[i)]
    \item  If $\vec{y}(t)=\vec{0}$ for all $t\in \N$, then $\lim_{t\rightarrow \infty} r(t)= \infty$.
    \item  If in addition $\sum_{t=0} 1/\rho(t)=\infty$ and $\lim_{t\rightarrow \infty} (\vec{x}(t),\vec{z}(t))=  (\vec{x}^{\star},\vec{z}^{\star})$, then $(\vec{x}^{\star},\vec{z}^{\star})$ satisfies the FON conditions of Problem~\eqref{eq:main_problem_formulation}.
 \end{enumerate}
\end{prop}
 \begin{proof}
 i) Note that Assumption~\ref{assumption:ADPM-unCon-2} implies that the sequence $\nabla f(\vec{x}(t))$ is bounded, when~\ref{assumption:ADPM-unCon-2}.a holds then the result is obvious and when~\ref{assumption:ADPM-unCon-2}.b holds the result follows from  Lemma~\ref{sec3ADPM:Lemma1}.
  In particular, there exists $M\in \R$ such that $||\nabla f(\vec{x}(t))||<M$ for all $t\in \N$.

  Using the FON conditions of the $\vec{x}$- and $\vec{y}$- updates we get
 \begin{align}
    \vec{0} &=   \nabla f(\vec{x}(t{+}1))+\rho(t)(\vec{x}(t{+}1)+\vec{B}\vec{z}(t)), \label{eq:ADPM-UCi-kkt-a} \\
    \vec{0} &=     \vec{B}\tran (\vec{x}(t{+}1)-\vec{B}\vec{z}(t{+}1)),  \label{eq:ADPM-UCi-kkt-b}
 \end{align}
 and rearranging~\eqref{eq:ADPM-UCi-kkt-a} and~\eqref{eq:ADPM-UCi-kkt-b}, we obtain
 \begin{align}
    \vec{z}(t) &{=}  (\vec{B}\tran \vec{B})^{-1} \vec{B}\tran \left( \vec{x}(t{+}1) {+} \frac{1}{\rho(t)} \nabla f(\vec{x}(t{+}1))\right), \label{eq:ADPM-UCi-kkt-c} \\
    \vec{z}(t{+}1) &{=}    (\vec{B}\tran \vec{B})^{-1}   \vec{B}\tran \vec{x}(t{+}1). \label{eq:ADPM-UCi-kkt-d}
 \end{align}  
  Using~\eqref{eq:ADPM-UCi-kkt-c},~\eqref{eq:ADPM-UCi-kkt-d}, and that $\nabla f(\vec{x}(t))$ is bounded  we get 
  \begin{align}
      ||\vec{z}(t{+}1){-}\vec{z}(t)|| {=}& \frac{1}{\rho(t)} \left|\left|   (\vec{B}\tran \vec{B})^{-1}   \vec{B}\tran   \nabla f(\vec{x}(t{+}1))) \right|\right| \label{eq:ADPM-UCi-kkt-e} \\
          {\leq}&   \frac{M}{\rho(t)} \left|\left|   (\vec{B}\tran \vec{B})^{-1}   \vec{B}\tran   \right|\right|.   \label{eq:ADPM-UCi-kkt-f}
  \end{align}
 Similarly, using~\eqref{eq:ADPM-UCi-kkt-a} and that $\nabla f(\vec{x}(t))$ is bounded  we get  
  \begin{align}
      ||\vec{x}(t{+}1)+\vec{B}\vec{z}(t)|| &=  \frac{1}{\rho(t)} ||\nabla f(\vec{x}(t{+}1))|| \leq \frac{M}{\rho(t)}. \label{eq:ADPM-UCi-kkt-g}
  \end{align}
  Finally, using~\eqref{eq:ADPM-UCi-kkt-f},~\eqref{eq:ADPM-UCi-kkt-g} and the triangle inequality gives
  \begin{align}
 \hspace{-0.2cm}     ||\vec{x}(t{+}1){+}\vec{B}\vec{z}(t{+}1)|| &{\leq} ||\vec{x}(t{+}1){+}\vec{B}\vec{z}(t)|| {+} || \vec{B}(\vec{z}(t{+}1) {-}\vec{z}(t))|| \nonumber \\
       &{\leq}    \frac{M}{\rho(t)}\Big(1 +  \left|\left| \vec{B}   (\vec{B}\tran \vec{B})^{-1}   \vec{B}\tran   \right|\right| \Big) . \label{eq:ADPM-UCi-kkt-h}
  \end{align}
   Since $ \rho(t)$ diverges to $\infty$,~\eqref{eq:ADPM-UCi-kkt-h} converges to zero, which concludes the proof.

 ii)  We need to show that $(\vec{x}^{\star},\vec{z}^{\star})$ satisfies the FON conditions (Definition~\ref{def:def_first_order_necessary_condition}) for Problem~\eqref{eq:main_problem_formulation} together with some Lagrangian multiplier. Note that condition 1) of the FON conditions  (Primal feasibility) holds because of part i) of this proposition  and conditions 2) and 3) of the FON conditions
   (dual feasibility and complementary slackness) trivially hold since there are no inequality constraints, since $\mathcal{X}=\R^{p_1}$ and $\mathcal{Z}=\R^{p_2}$. 
 Hence we only need to show condition 4) that the Lagrangian vanishes.
 We note that the gradient of the Lagrangian is 
 \begin{align}
   \nabla f(\vec{x}^{\star}) + \Blambda = \vec{0} ~~ \text{ and } ~~
   \vec{B}\tran \Blambda = \vec{0}, \label{eq:ADPM-UC-MP-ii-kkt}
 \end{align}
  where $\Blambda \in \R^n$ is the dual variable.
  If $\nabla f(\vec{x}^{\star})$ is in the null space of $\vec{B}\tran$ then~\eqref{eq:ADPM-UC-MP-ii-kkt} is satisfied by setting $\Blambda=-\nabla f(\vec{x}^{\star})$, which would conclude the proof.
 Therefore, in the sequel, we show that $\vec{B}\tran \nabla f(\vec{x}^{\star}){=}\vec{0}$.

 Using~\eqref{eq:ADPM-UCi-kkt-c} and~\eqref{eq:ADPM-UCi-kkt-d} gives
 \begin{align} \label{eq:ADPM-UCii-before-telescope}
    \sum_{t=0}^{\infty} (\vec{B}\tran \vec{B}) (\vec{z}(t{+}1)-\vec{z}(t)) {=}   \sum_{t=0}^{\infty} \frac{1}{\rho(t)} \vec{B}\tran \nabla f(\vec{x}(t{+}1)).
 \end{align}
 The left hand side of~\eqref{eq:ADPM-UCii-before-telescope} is a telescopic series, hence 
 \begin{align} \label{eq:ADPM-UCii-telescope}
      \sum_{t=0}^{\infty} \frac{1}{\rho(t)} \vec{B}\tran \nabla f(\vec{x}(t{+}1))=  (\vec{B}\tran \vec{B})(\vec{z}^{\star}-\vec{z}(0)),
 \end{align}
 which in turn ensures the convergence of both~\eqref{eq:ADPM-UCii-telescope} and 
 \begin{align} \label{eq:ADPM-UCii-telescope-after}
  \sum_{t=0}^{\infty} \frac{1}{\rho(t)} \left|\left| \vec{B}\tran \nabla f(\vec{x}(t{+}1)) \right|\right|.
 \end{align}
 Set $L=\lim_{t\rightarrow \infty} || B\tran\nabla f(\vec{x}(t))||= || B\tran \nabla f(\vec{x}^{\star})||$. 
  Let us next use contraction to show that $L=0$, which in turn shows that $B\tran\nabla f(\vec{x}^{\star})=\vec{0}$.
  Without of loss of generality, suppose $L>0$.
  Choose $\epsilon>0$ and $T\in \N$ such that $|| B\tran\nabla f(\vec{x}(t))||> L-\epsilon>0$ for all $t\geq T$.
 Then
\begin{align*}
       \sum_{t=0}^{\infty} \frac{1}{\rho(t)} \left|\left| \vec{B}\tran \nabla f(\vec{x}(t{+}1)) \right|\right| \geq&  
      \sum_{t=0}^{T-1} \frac{1}{\rho(t)} \left|\left| \vec{B}\tran \nabla f(\vec{x}(t{+}1)) \right|\right|  \\
           &+ (L-\epsilon) \sum_{t=T}^{\infty} \frac{1}{\rho(t)},
\end{align*}
 where  the right hand side diverges to $\infty$, since $\sum_{t=0}^{\infty} 1/\rho(t)=\infty$, which implies that the left hand side also diverges to $\infty$. This contradicts that the series~\eqref{eq:ADPM-UCii-telescope-after} converges and therefore we can conclude that $L=0$.
 \end{proof}

 \begin{remark}
  In Proposition~\ref{prop:ADPM-unCon-feas} we considered the case where $\vec{y}(t){=}\vec{0}$, which allowed us to derive the theoretical results. 
  Still, our numerical results in Section~\ref{sec:general applications} show that it can be beneficial to update $\vec{y}$ according to the recursion $\vec{y}(t{+}1){=}\vec{y}(t){+}\rho(\vec{x}(t{+}1){-}\vec{B}\vec{z}(t{+}1))$.
 \end{remark}

}

 \subsection{ Algorithm properties: Constrained Case} \label{sec:ADPM-reaching-feasibility}

In this section, we derive the convergence properties of the ADPM when $\mathcal{X}$ and $\mathcal{Z}$ are proper subsets of $\R^{p_1}$ and $\R^{p_2}$, respectively.
 Our convergence results assert that the primal feasibility of problem~\eqref{eq:main_problem_formulation}, which is a necessary optimality condition, is achieved as ADPM proceeds.
 More specifically, we show that regardless of whether $f,g$ are convex or nonconvex, whenever $\mathcal{X}$ and $\mathcal{Z}$ are convex, the primal residual at iteration $t$ of the ADPM (i.e., $\vec{A}\vec{x}(t)+\vec{B}\vec{z}(t)-\vec{c}$) converges to zero  as ADPM proceeds.
 To establish this result precisely, let us first make the following assumptions:
\begin{assumption} \label{assumption:ADPM-1}
  The functions $f$ and $g$ of problem~\eqref{eq:main_problem_formulation} are continuously differentiable.
\end{assumption}
\begin{assumption} \label{assumption:ADPM-2}
  The sets $\mathcal{X}$ and $\mathcal{Z}$ of problem~\eqref{eq:main_problem_formulation} are convex and compact.
\end{assumption}
\begin{assumption} \label{assumption:ADPM-3}
  Slater's condition~\cite{nonlinear_bertsekas} holds \emph{individually} for $\mathcal{X}$ and $\mathcal{Z}$. In particular, there exists a $\vec{x}\in\mathcal{X}$ (respectively, $\vec{z}\in\mathcal{Z}$) such that all the inequality constraints characterizing $\mathcal{X}$ (respectively, $\mathcal{Z}$) are inactive at $\vec{x}$ (respectively, $\vec{z}$).
\end{assumption}
\begin{assumption} \label{assumption:ADPM-4}
    The matrices ${\vec A}$ and $\vec B$ of problem~\eqref{eq:main_problem_formulation} have full column rank.
\end{assumption}
Note that we make no convexity assumptions on $f$ and $g$. However, the convexity assumption on $\mathcal{X}$ and $\mathcal{Z}$ is \emph{essential}.
 Otherwise, primal feasibility is not guaranteed in general, see Example~\ref{example:ADPM-fails} later in this section. Assumption~\ref{assumption:ADPM-3} is an additional technical
condition, similar to the constraint qualifications usually used in convex analysis. The last assumption is technically necessary to ensure that both ${\vec A}\tran {\vec A}$ and ${\vec B}\tran {\vec B}$ are positive definite. 
It is quite common in practice that this assumption holds, as desired, see Section~\ref{sec:general applications}. The following proposition establishes the convergence of ADPM:

\begin{prop} \label{prop:ADPM-feas}

   Suppose assumptions~\ref{assumption:ADPM-1}-\ref{assumption:ADPM-4} hold. 
 Let $\{\rho(t)\}_{t\in\N}$ 
 be a sequence of penalty parameters used in the ADPM algorithm, where $\rho(t+1)\geq \rho(t)$ for all $t$ and suppose there exists an integer $\kappa>0$ and a scalar $\Delta>1$ such that $\rho(t+\kappa) \geq \Delta \rho(t)$ for all $t$.
  Let $r(t)$ be the residual at iteration $t$ of the ADPM defined as $r(t)=||\vec{A}\vec{x}(t)+\vec{B}\vec{z}(t)-\vec{c}||$. Then $\lim_{t\rightarrow \infty} r(t)=0$.
\end{prop}

\begin{IEEEproof} Recall that $\{\vec{y}(t)\}_{t\in\N}$ is a bounded sequence. Thus, there exists $M_0>0$ such that $||\vec{y}(t)||\leq M_0$, for all $t\in \N$. We denote by $\mathcal{Y}$ the closed ball with radius $M_0$ centered at the origin $\vec{0}$, i.e., $\mathcal{Y}=\{ \vec{y} \in \R^q \big| ||\vec{y}|| \leq M_0\}$.



    Since $f$ and $g$ are continuous and the sets $\mathcal{X}$ and $\mathcal{Z}$ are compact there exists a scalar $M_1>0$ such that
   \begin{equation}\label{eq:M1}
       M_1 {=} \underset{(\vec{x},\vec{z},\vec{y}) {\in} \mathcal{X}{\times} \mathcal{Z}{\times} \mathcal{Y}}{\text{max}} |f(\vec{x}){+}g(\vec{z}){+}\vec{y}\tran ( \vec{A}\vec{x}{+}\vec{B}\vec{z}{-}\vec{c})|.
   \end{equation}
  In addition, $\hat{\vec{x}}:\R^{p_2} \rightarrow \R$ and $\hat{\vec{z}}:\R^{p_1} \rightarrow \R$, defined as
  \begin{align}
    \hat{\vec{x}}(\vec{z})= \underset{\vec{x} \in \mathcal{X}}{\text{argmin}} ||\vec{A}\vec{x}+\vec{B}\vec{z}-\vec{c}||^2, \label{eq:x-hat}\\
    \hat{\vec{z}}(\vec{x})= \underset{\vec{z} \in \mathcal{Z}}{\text{argmin}} ||\vec{A}\vec{x}+\vec{B}\vec{z}-\vec{c}||^2, \label{eq:z-hat}
  \end{align}
  are well-defined \emph{continuous functions} [compare with Assumption~\ref{assumption:ADPM-4}].
   By definition, $\vec{x}(t+1)$ is a solution of the optimization problem in $\vec{x}$-update of the ADPM.
   This, together with \eqref{eq:M1} yields 
\makeatletter%
\if@twocolumn%
   \begin{multline} \label{eq:ADPM-feas-proof-eq1}
     L_{\rho(t)}(\vec{x}(t+1),\vec{z}(t),\vec{y}(t)) \leq \\  M_1 + (\rho(t)/2)||\vec{A}\hec{x}(\vec{z}(t))+\vec{B}\vec{z}(t)-\vec{c}||^2  .
   \end{multline}
\else
   \begin{equation} \label{eq:ADPM-feas-proof-eq1}
     L_{\rho(t)}(\vec{x}(t+1),\vec{z}(t),\vec{y}(t)) \leq M_1 + (\rho(t)/2)||\vec{A}\hec{x}(\vec{z}(t))+\vec{B}\vec{z}(t)-\vec{c}||^2 .
   \end{equation}
\fi
\makeatother 
 
\noindent  Similarly, we get 
\makeatletter%
\if@twocolumn%
  \begin{multline} \label{eq:ADPM-feas-proof-eq2}
      L_{\rho(t)}(\vec{x}(t+1),\vec{z}(t+1),\vec{y}(t)) \leq  \\ \hspace{2.7mm} M_1 {+} (\rho(t)/2)||\vec{A}\vec{x}(t+1){+}\vec{B}\hec{z}(\vec{x}(t+1)){-}\vec{c}||^2.
  \end{multline}
\else
  \begin{equation} \label{eq:ADPM-feas-proof-eq2}
      L_{\rho(t)}(\vec{x}(t+1),\vec{z}(t+1),\vec{y}(t) \leq M_1 {+} (\rho(t)/2)||\vec{A}\vec{x}(t+1){+}\vec{B}\hec{z}(\vec{x}(t+1))-\vec{c}||^2.
  \end{equation}
\fi
\makeatother

 Let us first use \eqref{eq:ADPM-feas-proof-eq1} and \eqref{eq:ADPM-feas-proof-eq2} to derive a recursive relation for $r(t)$. By rearranging the terms of~\eqref{eq:ADPM-feas-proof-eq1} and by using that $|M_1-(f(\vec{x})+g(\vec{z})+ \vec{y}\tran ( \vec{A}\vec{x}{+}\vec{B}\vec{z}{-}\vec{c}) )|\leq 2M_1$ for all $(\vec{x},\vec{z},\vec{y}){\in} \mathcal{X} {\times} \mathcal{Z} {\times} \mathcal{Y}$, we have for all $t\in\N$,
 \begin{multline}
     ||\vec{A}\vec{x}(t+1)+\vec{B}\vec{z}(t)-\vec{c}||^2 \\ \leq \frac{4M_1}{\rho(t)} + ||\vec{A}\hec{x}(\vec{z}(t))+\vec{B}\vec{z}(t)-\vec{c}||^2 .\label{eq:ADPM_main_feasiblity_prop_ineq1}
\end{multline}
Moreover, we have for all $t\in\N$,
   \begin{align}
   \hspace{-2mm} r(t+1)
    & \leq \frac{4M_1}{\rho(t)} {+} ||\vec{A}\vec{x}(t{+}1){+}\vec{B}\hec{z}(\vec{x}(t{+}1)){-}\vec{c}||^2  \label{eq:ADPM_main_feasiblity_prop_ineq2}\\
    &  \leq \frac{8M_1}{\rho(t)} {+} ||\vec{A}\hec{x}(\vec{z}(t)){+}\vec{B}\vec{z}(t){-}\vec{c}||^2  \label{eq:ADPM_main_feasiblity_prop_ineq3}\\
    & \leq \frac{8M_1}{\rho(t)} {+} r(t)  \label{eq:ADPM-feas-proof-eq3} ,
\end{align}
 where \eqref{eq:ADPM_main_feasiblity_prop_ineq2} follows similarly by rearranging the terms of~\eqref{eq:ADPM-feas-proof-eq2} and by using that $|M_1-(f(\vec{x})+g(\vec{z})+ \vec{y}\tran ( \vec{A}\vec{x}{+}\vec{B}\vec{z}{-}\vec{c}) )|\leq 2M_1$ for all  $(\vec{x},\vec{z},\vec{y}){\in} \mathcal{X} {\times} \mathcal{Z} {\times} \mathcal{Y}$, \eqref{eq:ADPM_main_feasiblity_prop_ineq3} follows from combining the inequalities~\eqref{eq:ADPM_main_feasiblity_prop_ineq1} and~\eqref{eq:ADPM_main_feasiblity_prop_ineq2}, together with the definition of $\hec{x}$ and $\hec{z}$, and \eqref{eq:ADPM-feas-proof-eq3} follows by the definition of $\hec{x}$.

  Let us next use the recursive inequality~\eqref{eq:ADPM-feas-proof-eq3} above to show that $\{r(t)\}_{t\in\N}$ \emph{converges} to a \emph{finite} value. The  inequality \eqref{eq:ADPM-feas-proof-eq3} implies for all $t,n\geq 0$,
  \begin{equation}  \label{eq:rho0}
    r(t+n)  \leq  r(t) +  8M_1  \sum_{i=0}^{n-1} \frac{1}{\rho(t+i)} .
 \end{equation}
 From the definition of $\{\rho(t)\}_{t\in N}$, we get
 \begin{align}
     \sum_{i=0}^n \frac{1}{\rho(t+i) }&\leq \sum_{i=1}^{\lceil n/\kappa \rceil} \sum_{j=0}^{\kappa-1} \frac{1}{ \rho(t+i\kappa+j) }\label{eq:rho1}\\
     &\leq \sum_{i=0}^{\lceil n/\kappa \rceil} \frac{\kappa}{ \Delta^i \rho(t)} \label{eq:rho2}\\
     &\leq  \frac{\kappa}{\rho(t)} \sum_{i=0}^{\infty} \frac{1}{ \Delta^i }\label{eq:rho3},
 \end{align}
 where \eqref{eq:rho1} follows because the sum on the right contains all the terms of the sum on the left (and possibly more) and all the terms are positive, \eqref{eq:rho2} follows because  $1/\rho(t+i\kappa+j)\leq 1/ ( \Delta^i \rho(t))$ for all $0\leq j\leq \kappa-1$, and \eqref{eq:rho3} trivially follows from the nonnegativity of summands.
 Since $\Delta>1$, $\sum_{i=0}^{\infty} 1/ \Delta^i$ is a convergent geometric series, and thus let $\sum_{i=0}^{\infty} \kappa/ \Delta^i =M_2$. This, together with \eqref{eq:rho0}-\eqref{eq:rho3} implies that for all integers $t,n\geq 0$,
 \begin{equation}\label{eq:residual-convergence}
       r(t+n)  \leq r(t) +  \frac{8 M_1  M_2}{\rho(t)}  .
 \end{equation}
Now note that $\{r(t)\}_{t\in \N}$ is bounded.
 Moreover, because $\{\rho(t)\}_{t\in \N}$ is an increasing sequence, it follows that for all $\epsilon>0$, there exists a $T$ such that $(8M_1M_2/\rho(t))\leq \epsilon$, for all $t\geq T$. These, taken together with \eqref{eq:residual-convergence} and Lemma~\ref{lem:lem1} (see p.~\pageref{pp:lem1}), ensure that the sequence $\{r(t)\}_{t\in \N}$ \emph{converges} to a \emph{finite} value, denoted by $R$, i.e., $R=\lim_{t\rightarrow \infty} r(t)$.

Let us finally show that $R=0$.
  Since the set $\mathcal{X}\times \mathcal{Z}$ is compact, the sequence $\{\vec{x}(t),\vec{z}(t)\}_{t\in\N}$ has a limit point, say $(\bec{x},\bec{z})\in\mathcal{X}\times \mathcal{Z}$.
  Moreover, note that the function $||\vec A\vec x + \vec B\vec z -\vec c||^2$ is continuous on $\mathcal{X}\times \mathcal{Z}$.
  Therefore, taking limits as $t\rightarrow\infty$ in $r(t)=||\vec{A}\vec{x}(t)+\vec{B}\vec{z}(t)-\vec{c}||^2$, we have
  \be\label{eq:bar-xlimit}
  R{=}\lim_{t\rightarrow \infty} || \vec{A}\vec{x}(t){+}\vec{B}\vec{z}(t){-}\vec{c}||^2{=}||\vec{A}\bec{x}{+}\vec{B}\bec{z}{-}\vec{c}||^2 .
  \ee
 Let us now consider the limits in the inequality \eqref{eq:ADPM-feas-proof-eq3} as $t\rightarrow\infty$.
 Since $\lim_{t\rightarrow \infty} r(t+1) = \lim_{t\rightarrow \infty} ( (8M_1)/\rho(t) + r(t))=R$, from \eqref{eq:ADPM_main_feasiblity_prop_ineq2}, \eqref{eq:ADPM_main_feasiblity_prop_ineq3}, and the squeezing lemma, together with the continuity of functions $\hat{\vec x}$ and $\hat{\vec z}$ it follows that
\be\label{eq:x-hat-limit}
R= ||\vec{A}\hec{x}(\bec{z})+\vec{B}\bec{z}-\vec{c}||^2= ||\vec{A}\bec{x}+\vec{B}\hec{z}(\bec{x})-\vec{c}||^2.
\ee
 By combining \eqref{eq:bar-xlimit} and \eqref{eq:x-hat-limit}, together with the definitions \eqref{eq:x-hat} and \eqref{eq:z-hat}, we get
  \begin{align}
    \bec{x}= \underset{\vec{x} \in \mathcal{X}}{\text{argmin}} ||\vec{A}\vec{x}+\vec{B}\bec{z}-\vec{c}||^2,  \label{eq:ADPM-proof-KKT_x}\\
    \bec{z}= \underset{\vec{z} \in \mathcal{Z}}{\text{argmin}} ||\vec{A}\bec{x}+\vec{B}\vec{z}-\vec{c}||^2.  \label{eq:ADPM-proof-KKT_z}
  \end{align}
  Since Slater's constraint qualifications condition is satisfied for both sets $\mathcal{X}$ and $\mathcal{Z}$ (Assumption~\ref{assumption:ADPM-3}), $\bec{x}$ and $\bec{z}$ satisfy the first order necessary conditions
 for problems~\eqref{eq:ADPM-proof-KKT_x} and~\eqref{eq:ADPM-proof-KKT_z}, respectively.
  By combining these first order necessary conditions and~\eqref{eq:ADPM-proof-KKT_z}, it follows that $(\bec{x},\bec{z})$ satisfies the first order necessary conditions for the problem
  \begin{equation} \label{eq:APDM-proof-convex-problem}
    \begin{aligned}
      & \underset{\vec{x},\vec{z}}{\text{minimize}}
      & &||\vec{A}\vec{x}+\vec{B}\vec{z}-\vec{c}||^2\\
      & \text{subject to}
     & & (\vec{x},\vec{z}) \in \mathcal{X} \times \mathcal{Z}.
    \end{aligned}
  \end{equation}
  Since problem~\eqref{eq:APDM-proof-convex-problem} is convex and the constraint sets satisfy Slater's constraint qualifications condition, we conclude that $(\bec{x},\bec{z})$ is the solution to problem~\eqref{eq:APDM-proof-convex-problem}.
  Given that problem~\eqref{eq:main_problem_formulation} is feasible, we must have $||\vec{A}\bec{x}+\vec{B}\bec{z}-\vec{c}||^2=0$, and therefore $\lim_{t\rightarrow \infty} || \vec{A}\vec{x}(t)+\vec{B}\vec{z}(t)-\vec{c}||^2=0$ [compare with \eqref{eq:bar-xlimit}].
%
\end{IEEEproof}
\begin{lemma}\label{lem:lem1}
    \label{pp:lem1} Let us suppose that $\{a_t\}_{t\in \N}$ is a bounded sequence and for each $\epsilon>0$ there exists $T\in \N$ such that $a_{t+n}\leq \epsilon +a_t$ for all $n\geq 0$ and $t\geq T$, then $\lim_{t\rightarrow \infty} a_t$ exists.
  \end{lemma}
  \begin{IEEEproof}
         Let us denote by $R$ the limit inferior of $\{a_t\}_{t\in \N}$, i.e., $R=\liminf_{t\rightarrow \infty}a_t$, which is finite since $\{a_t\}_{t\in \N}$ is bounded.
          	It follows from elementary properties of the limit inferior that $\{a_t\}_{t\in \N}$ has a subsequence $\{a_{t_j}\}_{j\in \N}$ which converges to $R$ 
         , i.e., $\lim_{j\rightarrow \infty}a_{t_j} =R$.
         Subsequently, for a given $\epsilon>0$ we can find $J_1\in \N$ such that $|R-a_{t_j}|< \epsilon/2$ for all $j\geq J_1$.
         Moreover, by using the assumptions of the lemma, there exists $J_2\in \N$  such that $a_{t_j+n} \leq  \epsilon/2+ a_{t_j}$ for all $n\geq 0$ and $j\geq J_2$.
         If we choose $J=\max \{ J_1,J_2\}$ we get that $a_t \leq  \epsilon/2+ a_{t_J} < R + \epsilon$ for all $t\geq t_J$.
         Since this can be done for all $\epsilon>0$ we get that $\limsup_{t\rightarrow \infty} a_t \leq R$, implying that $\limsup_{t\rightarrow \infty} a_t=\liminf_{t\rightarrow \infty}a_t$.
       So we can conclude that $\lim_{t\rightarrow \infty} a_t = R$.
  \end{IEEEproof}

 One natural question that arises immediately with the Assumption~\ref{assumption:ADPM-2} is what if $\mathcal{X}$ and/or $\mathcal{Z}$ are \emph{nonconvex}. 
 The following example shows that the results of Proposition~\ref{prop:ADPM-feas} do generally not hold when either $\mathcal{X}$ or $\mathcal{Z}$ are nonconvex.
 \begin{example} \label{example:ADPM-fails}
  Consider the problem \label{page:PrimalInfeasibility}
\begin{equation} \label{eq:ADPM-fails-to-reach-feasible-point}
  \begin{aligned}
    & \underset{ x, z}{\mbox{ \emph{minimize}}}
    & & x^2+z^2 \\
    & \text{\emph{subject to}}
    & &  -2x+ z= 0.1, \\
    &&& x \in [-1,0]\cup [1,2], \hspace{0.2cm}
     z\in[0,3]
  \end{aligned}
\end{equation}
  The feasibility set and contours of the objective function are given in Fig.~\ref{fig:ADPM-fails}.
  It can be observed that if $z(0)=0$ and $\vec{y}(t)=\vec{0}$ for all $t\in \N$, then the optimal solution of both the $\vec{x}$- and $\vec{z}$- updates is $0$ for all $t\in\N$, i.e., $\lim_{t\rightarrow \infty} {x}(t)=0$ and $\lim_{t\rightarrow \infty} {z}(t)=0$. This means that the algorithm converges to $(0,0)$, which is an infeasible point.
 \end{example}
\begin{figure}[t]
\centering
{\includegraphics[width=0.30\textwidth]{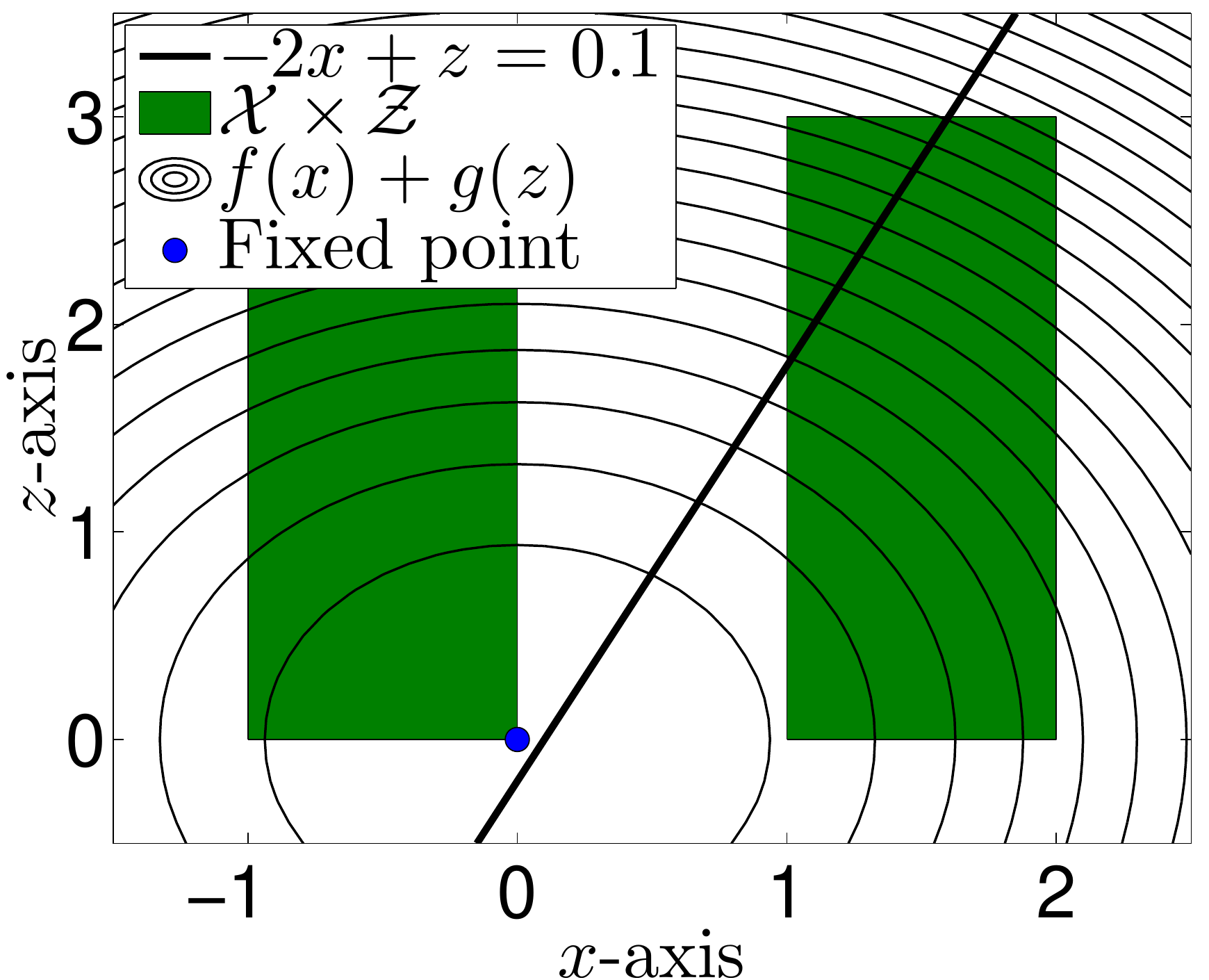}
}
\caption{Example where ADPM fails to converge to a feasible point when sets $\mathcal{X}$ and $\mathcal{Z}$ are nonconvex.}
\label{fig:ADPM-fails}
\vspace{-4mm}
\end{figure}
Note that our Assumption~\ref{assumption:ADPM-1} is a weaker condition than assuming that  $f$ and/or $g$ are convex. 
 As a result, characterizing generally the proprieties of the objective value of ADPM after the convergence is technically challenging.
  Nevertheless, ADPM appears to resemble a sequential optimization approach, which provides degrees of freedom to hover over the true objective function for locating a \emph{good} objective value.
 In~\cite{Magnusson_2014_2}, we give some experiments to numerically show these appealing aspects of the ADPM, besides those ensured by Proposition~\ref{prop:ADPM-feas}.

\section{Alternating Direction Method of Multipliers} \label{sec:ADMM}
  In this section we investigate some new general properties of the ADMM in a nonconvex setting.
  We state the algorithm in section~\ref{sec:ADMM-alg-descript} and study convergence properties in section~\ref{sec:ADMM-analytic-properties}. 

 \subsection{Algorithm Description} \label{sec:ADMM-alg-descript}
  The ADMM can explicitly be stated as follows.

{\tableANDalgSize

\noindent\rule{\linewidth}{0.3mm}
\\
\emph{Algorithm 2: \ \textsc{The Alternating Direction Method of Multipliers (\small{ADMM}) }}

\vspace{-0.2cm}
\noindent\rule{\linewidth}{0.3mm}
\begin{enumerate}
   \item \label{Alg:ADMM-initialization-step}  \textbf{Initialization:} Set $t=0$ and put initial values to $\vec{z}(t)$, $\vec{y}(t)$, and $\rho$.
   \item \label{Alg:ADMM-x-step} \textbf{\vec{x}-update:}  $\vec{x}(t{+}1)=\underset{\vec{x}\in \mathcal{X}}{\text{argmin}}~  L_{\rho}(\vec{x},\vec{z}(t),\vec{y}(t))$.
   \item \label{Alg:ADMM-z-step} \textbf{\vec{z}-update:} $\vec{z}(t{+}1)=\underset{\vec{z}\in \mathcal{Z}}{\text{argmin}}~L_{\rho}(\vec{x}(t{+}1),\vec{z},\vec{y}(t))$.
   \item \label{Alg:ADMM-y-step} \textbf{\vec{y}-update:} $\vec{y}(t{+}1)= \vec{y}(t) + \rho  (\vec{A}\vec{x}(t{+}1){+}\vec{B}\vec{z}(t+1)-\vec{c})$.
   \item \label{Alg:ADMM-stop} \textbf{Stopping criterion:} If stopping criterion is met terminate, otherwise set $t=t+1$ and go to step 2.
\end{enumerate}
\vspace{-3mm}
\rule{\linewidth}{0.3mm}

}

 \noindent


%
%



 \BlueText{Unlike in Algorithm~1 (ADPM), in Algorithm~2 (ADMM) the penalty parameter is fixed. }
  The first step is the initialization (step~\ref{Alg:ADMM-initialization-step}).
  As presented above, the $\vec{x}$- and $\vec{z}$- updates require a \emph{solution} of an optimization problem.
 This is not as restrictive as it may seem, since under mild conditions such requirements are accomplished, see Examples~\ref{example:subproblem-become-strongly-convex}-\ref{example:subproblem-become-strongly-convex-3}. However, we note that no such global optimality requirement of $\vec x(t+1)$ and $\vec z(t+1)$ is necessary in our convergence assertions, as we will show in subsequent sections. More specifically, our convergence results apply as long as $\vec x(t+1)$ [respectively, $\vec z(t+1)$] is a local minimum. 

   \subsection{Algorithm Properties} \label{sec:ADMM-analytic-properties}

 In this section, we show that, under mild assumptions, if the sequence $\{\vec y(t)\}_{t\in\N}$ converges to some $\bar{\vec y}$, then any limit point of the sequence $\{\vec x(t),\vec z(t)\}_{t\in\N}$, together with $\bar{\vec y}$, satisfy FON conditions of Problem~\eqref{eq:main_problem_formulation}~(compare with Definition~\ref{def:def_first_order_necessary_condition}). It is worth noting that these results hold regardless of whether $f$, $g$, $\mathcal{X}$, and $\mathcal{Z}$ are convex or nonconvex.

 Let us now scrutinize the above assertion precisely. The analysis is based on the following assumption which can be expected to hold for many problems of practical interest:

\begin{assumption} \label{assumption:on_sets_X_and_Z}
   The sets $\mathcal{X}$ and $\mathcal{Z}$ of problem~\eqref{eq:main_problem_formulation} are closed and can be expressed in terms of a finite number of equality and inequality constraints. In particular,
  \begin{align*}
     \mathcal{X} &= \{\vec{x}\in \R^{p_1} ~   \big| ~ \Bo\psi(\vec{x})=\vec{0}, ~~\Bo\phi(\vec{x})\leq\vec{0}\}, \\
     \mathcal{Z} &= \{\vec{z}\in \R^{p_2} ~ \big| ~ \Bo\theta(\vec{z})=\vec{0},~~~ \Bo\sigma(\vec{z})\leq\vec{0}\},
  \end{align*}
   where $\Bo\psi:\R^{p_1}\rightarrow \R^{q_1}$, $\Bo\phi:\R^{p_1}{\rightarrow} \R^{q_2}$, $\Bo\theta:\R^{p_2}{\rightarrow} \R^{q_3}$, and $\Bo\sigma:\R^{p_2}{\rightarrow} \R^{q_4}$ are continuously differentiable functions.
\end{assumption}

\begin{assumption} \label{assumption:local_optimality}
        For every $t\in\N$, $\vec x(t)$ [respectively, $\vec z(t)$] computed at step~$2$ (respectively, step~$3$) of the ADMM algorithm is locally \emph{or} globally optimal. 
\end{assumption}
\begin{assumption} \label{assumption:ADMMregularity}
         Let $\mathcal{L}$ denote the set of limit points of the sequence $\{(\vec{x}(t),\vec{z}(t))\}_{t\in \N}$ and let $(\bec{x},\bec{z})\in\mathcal{L}$. The set of constraint gradient vectors at $\bec{x}$,
%
        \be  \label{eq:assumpt9} 
        \mathcal{C}_{\mathcal{X}}(\bec{x}){=} \{ {\nabla} \Bo\psi_i(\bec{x}) | i{=}1,{\cdots}, q_1\}  {\cup}
   \{ \nabla \Bo\phi_i(\bec{x}) | i {\in} \mathcal{A}_{\mathcal{X}}(\bec{x}) \}
        \ee
        associated with the set $\mathcal{X}$ is linearly independent, where $\mathcal{A}_{\mathcal{X}}(\bec{x})=\{ i \ | \  \Bo\phi_i(\bec{x}) = 0 \}$. Similarly, the corresponding set of constraint gradient vectors $\mathcal{C}_{\mathcal{Z}}$ associated with the set $\mathcal{Z}$ is linearly independent.
\end{assumption}
 Assumption~\ref{assumption:on_sets_X_and_Z} is self-explanatory. Note that steps~$2$ and $3$ of the algorithm involve nonconvex optimization problems, where the computational cost of finding the solutions $\vec x(t+1)$ and $\vec z(t+1)$, in general, can be entirely prohibitive.
 However, Assumption~\ref{assumption:local_optimality} indicates that the solution $\vec{x}(t+1)$ [respectively, $\vec z(t+1)$] of the optimization problem associated with the steps~$2$ (respectively, $3$) of the ADMM should \emph{only} be a \emph{local minimum} and not necessarily a global minimum.
 Thus, Assumption~\ref{assumption:local_optimality} can usually be accomplished by employing efficient local optimization methods (see \cite[Section~1.4.1]{convex_boyd}). In the literature, Assumption~\ref{assumption:ADMMregularity} is called the ``regularity assumption" and is usually satisfied in practice. Moreover, any point that complies with the assumption is called regular, see~\cite[p.~269]{nonlinear_bertsekas}. Let us next document two results that will be important later.
\begin{lemma} \label{remark:independence}
        Suppose Assumptions~\ref{assumption:on_sets_X_and_Z} and \ref{assumption:ADMMregularity} hold.
%
%
 Let $\{ (\vec{x}(t_k),\vec{z}(t_k)) \}_{k\in \N}$ be a subsequence of $\{ (\vec{x}(t),\vec{z}(t)) \}_{t\in \N}$ with $\lim_{k\rightarrow \infty} (\vec{x}(t_k), \vec{z}(t_k))=  (\bec{x}, \bec{z})$.
 Then there exists $K$ such that the sets of vectors $\mathcal{C}_{\mathcal{X}}(\vec{x}(t_k))$ and $\mathcal{C}_{\mathcal{Z}}(\vec{z}(t_k))$ [cf~\eqref{eq:assumpt9}] are each linearly independent for all $k\geq K$. 
\end{lemma}
\begin{proof}
 First note that if $i\notin \mathcal{A}(\bec{x})$, then $\Bo\phi_i(\vec{x}(t_k))<0$ [or $i \notin \mathcal{A}(\vec{x}(t_k))]$ for all sufficiently large $k$, since $\Bo\phi_i$ is continuous and the set $\{x\in \R|x\neq 0\}$ is open.
 Therefore, it suffices to show that the columns of the matrix $\vec{D}(\vec{x}(t_k))\in \R^{p_1\times (q_1+ |\mathcal{A}(\bec{x})|)}$ are linearly independent  for all sufficiently large $k$, where
  \begin{equation} \label{eq:Remark_1_D_matrix}
     \vec{D}(\vec{x}) = [(\nabla \Bo\psi_i(\vec{x}))_{i{=}1,{\cdots}, q_1}, ( \nabla \Bo\phi_i(\vec{x}) )_{i\in\mathcal{A}_{\mathcal{X}}(\bec{x})}].
  \end{equation}
 Since $\operatorname{Det}\big(\vec{D}(\vec{x})\tran \vec{D}(\vec{x})\big)$ is continuous (see Assumption~\ref{assumption:on_sets_X_and_Z}),
$\operatorname{Det}\big(\vec{D}(\vec{x}(t_k))\tran\vec{D}(\vec{x}(t_k))\big)$ can be made arbitrarily close to $\operatorname{Det}\big(\vec{D}(\bec{x})\tran \vec{D}(\bec{x})\big)$, which is \emph{nonzero}, see Assumption~\ref{assumption:ADMMregularity}.
 Equivalently, there exists $K\in \N$ such that $\operatorname{Det}\big(\vec{D}(\vec{x}(t_k))\tran\vec{D}(\vec{x}(t_k))\big)$ is nonzero for all $k\geq K$, which in turn ensures that $\mathcal{C}_{\mathcal{X}}(\vec{x}(t_k))$ is a linearly independent set for $k\geq K$. The linear independence of  $\mathcal{C}_{\mathcal{Z}}(\vec{z}(t_k))$ for all sufficiently large $k$ can be proved similarly.
%
\end{proof}
%
\begin{lemma} \label{assumption:conciseFON}
        Suppose  Assumptions~\ref{assumption:ADPM-1}, \ref{assumption:on_sets_X_and_Z}, \ref{assumption:local_optimality}, and \ref{assumption:ADMMregularity} hold.
  Let $\{ (\vec{x}(t_k),\vec{z}(t_k)) \}_{k\in \N}$ be a subsequence of $\{ (\vec{x}(t),\vec{z}(t)) \}_{t\in \N}$ with $\lim_{k\rightarrow \infty} (\vec{x}(t_k), \vec{z}(t_k))=  (\bec{x}, \bec{z})$.
 Then for sufficiently large $k$, there exist Lagrange multipliers $(\Bo\lambda(t_k), \Bo\gamma(t_k))\in \R^{q_1}\times \R^{q_2}$ \emph{[respectively, $(\Bo\mu(t_k),$ $\Bo\omega(t_k))$ $\in \R^{q_3}\times \R^{q_4}$]} such that the pair $\vec{x}(t_k)$, $(\Bo\lambda(t_k), \Bo\gamma(t_k))$ \emph{[respectively, $\vec{z}(t_k)$,  $(\Bo\mu(t_k), \Bo\omega(t_k))$]} satisfies the FON conditions of the optimization problem in the $\vec{x}$- \emph{(respectively, $\vec{z}$-)} update of the ADMM algorithm (compare with Definition~\ref{def:def_first_order_necessary_condition}).
\end{lemma}
\begin{proof}
  From Lemma~\ref{remark:independence}, we have that $\vec{x}(t_k)$ and $\vec{z}(t_k)$ are {regular} for sufficiently large $k$. This combined with the assumptions yields the result, which is an immediate consequence of~\cite[Proposition~3.3.1]{nonlinear_bertsekas}
\end{proof}

Lemmas~\ref{remark:independence} and \ref{assumption:conciseFON} play a central role when deriving our convergence results, as we will show in the sequel. The following proposition establishes the convergence results of the ADMM algorithm:

\begin{prop}\label{prop:ADMM-main-prop}
    Suppose the Assumptions~\ref{assumption:ADPM-1}, \ref{assumption:on_sets_X_and_Z}, \ref{assumption:local_optimality}, and \ref{assumption:ADMMregularity} hold and the sequence $\vec{y}(t)$ converges to a point, i.e., $\lim_{t\rightarrow \infty} \vec{y}(t){=}\bec{y}$ for some $\bec{y}$.
    Then every limit point of the sequence $\{\vec{x}(t),\vec{z}(t)\}_{t\in \N}$, together with $\bec{y}$ and some $\Bo\lambda {\in}\R^{q_1}$, $\Bo\gamma {\in}\R^{q_2}$, $\Bo\mu{\in}\R^{q_3}$, and $\Bo\omega {\in} \R^{q_4}$ satisfy the FON conditions of Problem~\eqref{eq:main_problem_formulation}.
\end{prop}
\begin{proof}
   Let $(\bec{x},\bec{z})$ be a limit point of $\{(\vec{x}(t),\vec{z}(t))\}_{t\in \N}$ and $\{(\vec{x}(t_k),\vec{z}(t_k))\}_{k\in \N}$ be a subsequence such that
   $\lim_{k\rightarrow \infty} (\vec{x}(t_k),\vec{z}(t_k)) = (\bec{x},\bec{z})$.
     We show that the primal variables $\bec{x},\bec{z}$ and the Lagrange multipliers $\bec{y}$, $\Bo\lambda$, $\Bo\gamma$, $\Bo\mu$, and $\Bo\omega$ satisfy the first order necessary conditions, where $\Bo\lambda$, $\Bo\gamma$, $\Bo\mu$, and $\Bo\omega$ are chosen as in Lemma~\ref{Lemma:In-ADMM-proof}.



  In the sequel, we show that the four conditions of Definition~\ref{eq:def_first_order_necessary_condition} (First order necessary condition) are all satisfied.

  1) Primal feasibility: Since $(\vec{x}(t_k),\vec{z}(t_k))\in\mathcal{X}\times \mathcal{Z}$ and the set $\mathcal{X}\times \mathcal{Z}$ is closed it follows that $(\bec{x},\bec{z}) \in  \mathcal{X}\times \mathcal{Z}$.
    Since $\bec{y} {=}\vec{y}(0)+\sum_{t=1}^{\infty} \rho (\vec{A}\vec{x}(t)+\vec{B}\vec{z}(t)-\vec{c})$, we must have $\lim_{t\rightarrow \infty} ||\vec{A}\vec{x}(t)+\vec{B}\vec{z}(t)-\vec{c}||^ 2=0$, or   $\vec{A}\bec{x}+\vec{B}\bec{z}=\vec{c}$.

  2) Dual feasibility: It holds for $\Bo\gamma (t_k)$ and $\Bo\omega(t_k)$ from Lemma~\ref{assumption:conciseFON} that $\Bo\gamma (t_k)\geq \vec{0}$ and $\Bo\omega(t_k) \geq \vec{0}$ (compare with Definition~\ref{def:def_first_order_necessary_condition}).  Hence, since the closed right half-plane is a closed set, it follows that $\Bo\gamma\geq \vec{0}$ and $\Bo\omega\geq \vec{0}$.

  3) Complementary slackness:  If $\Bo \phi_i (\bec{x})=0$ then $\Bo \gamma_i \Bo \phi_i (\bec{x}) = 0$ trivially holds.
                                                  On the other hand, if $\Bo \phi_i (\bec{x})<0$ then we showed in the proof of lemma~\ref{Lemma:In-ADMM-proof} that $\Bo \gamma_i = 0$.
    Hence, it follows that  $\Bo \gamma_i \Bo \phi_i (\bec{x}) = 0$.

 4) Lagrangian vanishes:  We need to show that
      \begin{align}
          \nabla_{\vec{x}} f(\bec{x})+\vec{A}\tran \bec{y} + \nabla_{\vec{x}}  \Bo\psi(\bec{x})  \Bo\lambda + \nabla_{\vec{x}}  \Bo\phi(\bec{x}) \Bo\gamma = \vec{0}, \label{eq:ADMMproof_LagVan_1}\\
          \nabla_{\vec{z}}  g(\bec{z})+\vec{B}\tran \bec{y} + \nabla_{\vec{z}}  \Bo\theta(\bec{z})  \Bo\mu + \nabla_{\vec{z}}  \Bo\sigma(\bec{z}) \Bo\omega = \vec{0}.  \label{eq:ADMMproof_LagVan_2}
      \end{align}
      Let us start by showing~\eqref{eq:ADMMproof_LagVan_2}.
      From Lemma~\ref{assumption:conciseFON}, we get for all sufficiently large $k$ that [compare with Definition~\ref{def:def_first_order_necessary_condition}]:
     \begin{multline} \label{eq:ADMM_KKT_proof_step_4_1}
          \nabla_{\vec{z}} L(\vec{x}(t_k),\vec{z},\vec{y}(t_k{-}1))  + \nabla_{\vec{z}}  \Bo\theta (\vec{z}(t_k)) \Bo\mu(t_k)  \\ + \nabla_{\vec{x}}  \Bo\sigma(\vec{x}(t_k)) \Bo\omega(t_k) {=} \vec{0}.
      \end{multline}
      By using $\vec{y}(t_k-1)=\vec{y}(t_k)-\rho(\vec{A}\vec{x}(t_k)+\vec{B}\vec{z}(t_k)-\vec{c})$ in equation~\eqref{eq:ADMM_KKT_proof_step_4_1} and rearranging the terms we get that
     \begin{multline}
          \nabla_{\vec{z}} g(\vec{z}(t_k))+\vec{B}\tran \vec{y}(t_k) +\nabla_{\vec{z}}  \Bo\theta (\vec{z}(t_k))  \Bo\mu(t_k) \\ {+} \nabla_{\vec{x}}  \Bo\sigma(\vec{x}(t_k)) \Bo\omega(t_k) {=} \vec{0}.
     \end{multline}
      By using that $\lim_{k\rightarrow \infty} (\vec{x}(t_k),\vec{z}(t_k),\vec{y}(t_k)) = (\bec{x},\bec{z},\bec{y})$ and  $\lim_{k\rightarrow \infty} ( \Bo\lambda(t_k),\Bo\gamma(t_k),  \Bo\mu(t_k), \Bo\omega(t_k)) = ( \Bo\lambda,\Bo\gamma,  \Bo\mu, \Bo\omega)$, we conclude that equation~\eqref{eq:ADMMproof_LagVan_2} holds.
     By using the same arguments as above we get  for all sufficiently large $k$ that  
     \begin{multline}
          \nabla_{\vec{x}} f(\vec{x}(t_k))+\vec{A}\tran \vec{y}(t_k) +\nabla_{\vec{x}}  \Bo\psi(\vec{x}(t_k))  \Bo\lambda(t_k) \\ {+} \nabla_{\vec{x}}  \Bo\phi(\vec{x}(t_k)) \Bo\gamma(t_k) = \rho \vec{A}\tran \vec{B} (\vec{z}(t_k){-}\vec{z}(t_k-1)).
     \end{multline}
     Therefore, by the arguments above, if we can show that $\lim_{t\rightarrow \infty}\rho \vec{A}\tran \vec{B} (\vec{z}(t{+}1){-}\vec{z}(t))=\vec{0}$, then equation~\eqref{eq:ADMMproof_LagVan_1} holds.
     The assumption $\bec{y}=\lim_{t\rightarrow \infty} \vec{y}(t)$ together with the relation $\vec{y}(t+1) =\vec{y}(0)+\rho \sum_{l=1}^{t+1} \vec{A}\vec{x}(l)+\vec{B}\vec{z}(l)-\vec{c}$ can be used to show that the series
      \begin{equation*}
             {\sum_{t=1}^{\infty}}  (\vec{A}\vec{x}(t){+}\vec{B}\vec{z}(t{+}1){-}\vec{c}),~~
             {\sum_{t=1}^{\infty}}  (\vec{A}\vec{x}(t){+}\vec{B}\vec{z}(t){-}\vec{c}),
    \end{equation*}
    are convergent.
   By taking the difference of the two series 
   and using that the sum of convergent series is a convergent series, we get that $\sum_{t=1}^{\infty} \vec{B}(\vec{z}(t{+}1)-\vec{z}(t))$ is a convergent series. Thus, implying that $\lim_{t\rightarrow \infty} \vec{B}(\vec{z}(t{+}1)-\vec{z}(t)) =0$.
   By multiplying $\rho \vec{A}\tran$ from the left side we get that $\lim_{t\rightarrow \infty}  \rho \vec{A}\tran \vec{B}(\vec{z}(t{+}1)-\vec{z}(t)) =0$.
\end{proof}

  \begin{lemma} \label{Lemma:In-ADMM-proof}
      Let $\{t_k\}_{k\in\N}$ be a sequence such that $\lim_{k\rightarrow \infty} (\vec{x}(t_k),\vec{z}(t_k)) = (\bec{x},\bec{z})$. Then the limits: $\underset{k\rightarrow \infty}{\lim} \Bo\lambda(t_k)$, $\underset{k\rightarrow \infty}{\lim}  \Bo\gamma(t_k)$, $\underset{k\rightarrow \infty}{\lim}  \Bo\mu(t_k)$, and $\underset{k\rightarrow \infty}{\lim}  \Bo\omega(t_k)$ exist, where $\Bo\lambda(t_k)$, $\Bo\gamma(t_k)$, $\Bo\mu(t_k)$, and $\Bo\omega(t_k)$ chosen as in Lemma~\ref{assumption:conciseFON}.
  \end{lemma}
   \begin{proof}
         We prove the existence of the first two limits. The proof of the existence of the latter two limits follows similarly.

       Since $\nabla f$, $\nabla \Bo\psi$, and $\nabla \Bo\phi$ are continuous functions (see Assumption~\ref{assumption:ADPM-1}) we have
      \begin{align*}
             \lim_{k\rightarrow \infty} \nabla f(\vec{x}(t_k)) = \nabla f(\bec{x}), ~~~~
            \lim_{k\rightarrow \infty}  \nabla \Bo\psi (\vec{x}(t_k)) = \nabla \Bo\psi(\bec{x}),\\
          \text{and}~~~~~~~~~~~~~   \lim_{k\rightarrow \infty} \nabla \Bo\phi(\vec{x}(t_k)) = \nabla \Bo\phi(\bec{x}).
      \end{align*}
 This, together with Lemma~\ref{remark:independence} implies that there exists $K$ such that $\vec{D}(\vec{x}(t_k))\tran  \vec{D}(\vec{x}(t_k))$ (see Eq.~\eqref{eq:Remark_1_D_matrix}) is invertible for all $k\geq K$.  Hence, it follows that for all $k\geq K$, we have
%
%
  \begin{multline*}
          (\Bo\lambda(t_k),(\Bo\gamma_i(t_k))_{i\in \mathcal{A}_{\mathcal{X}}(\bec{x})})= \\
             \vec{D}(\hspace{-0.2mm}\vec{x}(t_k)\hspace{-0.4mm})\tran  \vec{D}(\hspace{-0.2mm}\vec{x}(t_k)\hspace{-0.4mm})^{-1} \hspace{-0.3mm}  \vec{D}(\hspace{-0.2mm}\vec{x}(t_k)\hspace{-0.4mm})\tran (\nabla f(\vec{x}(\hspace{-0.2mm}t_k\hspace{-0.3mm})\hspace{-0.4mm})+\vec{A}\tran \vec{x}(\hspace{-0.2mm}t_k\hspace{-0.3mm})\hspace{-0.4mm}).
  \end{multline*}
    Since $\vec{D}(t_k)$ and $\nabla f(\vec{x}(t_k))$  converge when $k\rightarrow \infty$ it follows that $\lim_{k\rightarrow \infty}(\Bo\lambda(t_k),(\Bo\gamma_i(t_k))_{i\in \mathcal{A}_{\mathcal{X}}(\bec{x})})$ exists.

   Next we show that $\lim_{k\rightarrow \infty}\Bo\gamma_i(t_k)=0$ if $i\notin \mathcal{A}_{\mathcal{X}}(\bec{x})$.
    Since $\Bo \phi_i (\bec{x})<0$, there exists an open set $\mathcal{U}\subseteq \R^{p_2}$ containing $\bec{x}$ such that $\Bo \phi_i (\vec{x})<0$ for all $\vec{x}\in \mathcal{U}$.
    In particular, there exists $K\in \N$ such that $\Bo \phi_i (\vec{x}(t_k))<0$ for $k\geq K$.
  Therefore, there must exist $K{\in} \N $ such that $\Bo \gamma_i(t_k)=0$ for all $k \geq K$, since complementary slackness [$\Bo \gamma_i(t_k) \Bo \phi_i (\vec{x}(t_k))=0$] holds for all sufficiently large $k$ [compare with Lemma~\ref{assumption:conciseFON}].\end{proof}
%
%
%
%
%
%


A stronger version of Proposition~\ref{prop:ADMM-main-prop} is shown in the following corollary:
 \begin{corollary}
     If $\lim_{t\rightarrow}(\vec{x}(t),\vec{z}(t),\vec{y}(t))=(\bec{x},\bec{z},\bec{y})$, then $\bec{x}$ and $\bec{z}$ satisfy the FON conditions of Problem~\eqref{eq:main_problem_formulation}.
 \end{corollary}
\noindent The corollary follows immediately because the hypothesis implies that the set $\mathcal{L}$ defined in Assumption~\ref{assumption:ADMMregularity} is a singleton.

 Technically, Proposition~\ref{prop:ADMM-main-prop} characterizes the solution of the ADMM algorithm applied on the possibly nonconvex problem~\eqref{eq:main_problem_formulation}. 
 More specifically, the proposition claims that under mild assumptions the solutions computed by ADMM satisfy the FON conditions for problem~\eqref{eq:main_problem_formulation}, if at every iteration, the subproblems are locally (or globally) solved and if the dual variables of ADMM converge. 

\BlueText{
Let us now show how Proposition~\ref{prop:ADMM-main-prop} can be used to completely characterize the convergence of the ADMM for a class of problems identified by the following assumption.
\begin{assumption} \label{ADMM-scalar-objs}
       $f,g : \R \rightarrow \R$, $\mathcal{X}{=}\mathcal{Y}{=}\R$, and the coupling constraint is $x{=}z$, i.e., $\vec{A}{=}1$ and $\vec{B}{=}{-}1$. 
        In addition, the derivatives $f'$ and $g'$ are $L$-Lipschitz continuous. 
\end{assumption}

\noindent
 The following corollary of Proposition~\ref{prop:ADMM-main-prop} shows that under Assumption~\ref{ADMM-scalar-objs} the ADMM always either converges or diverges to $\pm \infty$ and characterizes the convergence in terms of $z(0)$.
 \begin{corollary} \label{Prop:SCA-ADMM}
   Suppose Assumption~\ref{ADMM-scalar-objs} holds, $\rho>L$, and $y(0){=}g'(z(0))$.  
   Then  $$\lim_{k\rightarrow \infty }(x(k),z(k),y(k))=(z^{\star},z^{\star},g'(z^{\star})),$$
    where $z^{\star}$ is determined as follows:
   \begin{enumerate}[a)] 
       \item If $f'(z(0))+g'(z(0)) = 0$, then $z^{\star}=z(0)$.
       \item If $f'(z(0))+g'(z(0)) < 0$, then
                $$z^{\star}=\inf \{ z\geq z(0) | f'(z)+g'(z)=0\}.$$
       \item If $f'(z(0))+g'(z(0)) > 0$, then
                $$z^{\star}=\sup \{ z\leq z(0) | f'(z)+g'(z)=0\}.$$
   \end{enumerate}
 \end{corollary}
 \begin{proof}
   We start by writing the steps of the ADMM in a more convenient form.
 Note that  $ g'(z(t{+}1)) +y(t)+\rho(x(t{+}1)-z(t{+}1)) = 0$, from the optimality  conditions of $z(t{+}1)$ at the $\vec{z}$-update.
 This combined with the $\vec{y}$-update yields (i) $y(t) = g'(z(t))$. 
 Moreover, because $f'$ and $g'$ are L-Lipschitz continuous, we have that (ii) the functions $L_{\rho}(\cdot,z(t) , y(t))$ and $L_{\rho}(x(t),\cdot, y(t))$, associated with the $\vec{x}$- and $\vec{z}$- updates are strongly convex for all $\rho>L$. 

   From (i) and (ii), we get that $x(t{+}1)$ is the unique solution to
   \begin{align}
      0=&f'(x)+g'(z(t))+ \rho(x-z(t)), \label{eq:ADMM-SCA-proof-x-update}
    \end{align}
    and $z(t{+}1)$ is the unique solution to
    \begin{align}
    0=&g'(z)-g'(z(t))- \rho(x(t{+}1)-z). \label{eq:ADMM-SCA-proof-z-update}
    \end{align}
  In the sequel, we show each case a), b), and c) separately.

      a)   If $f'(z(0))+g(z(0))=0$ then $x(t)=z(0)$ and $z(t)=z(0)$ are clearly the unique solutions to~\eqref{eq:ADMM-SCA-proof-x-update} and~\eqref{eq:ADMM-SCA-proof-z-update}, respectively, for all $t\geq 1$.  The result follows.

      b)  In the sequel, we show that $z(t{+}1){>}z(t)$ and $z(t){<}z^{\star}$ for all $t\in \N$, implying that $\bar{z}=\lim_{t\rightarrow \infty} z(t)$ exists (it is possible that $\bar{z}{=}\infty$ when $z^{\star}{=}\infty$). 
                Since the interval $\mathcal{U}= [z(0),z^{\star}]$ (or $\mathcal{U}=[z(0),z^{\star}[$ when $z^{\star}=\infty$) 
                 is a closed set, $\bar{z}\in \mathcal{U}$. 
                Moreover, by Proposition~\ref{prop:ADMM-main-prop},  $(x(t),z(t),y(t))$ can only converge to a point satisfying the first order necessary conditions, i.e., to a point $(z,z,g'(z))$ with $f'(z)+g'(z)=0$. 
          When $z^{\star}<\infty$, the only $z\in \mathcal{U}$ satisfying the necessary conditions is $z^{\star}$ and when $z^{\star}=\infty$ no  $z\in \mathcal{U}$ satisfies the necessary conditions. 
           Hence, we can conclude that $\bar{z}=z^{\star}$.         

          We show that $z(t{+}1){>}z(t)$ and $z(t{+}1)<z^{\star}$ for all $z(t)\in [z(0),z^{\star}[$, 
          but as an intermediary step we first show that $x(t{+}1){>}z(k)$ and $x(t{+}1)<z^{\star}$ for all $z(t)\in [z(0),z^{\star}[$.   
          To see that $x(t{+}1){>}z(k)$, we note that $x(t{+}1)\leq z(k)$ contradicts the $L$-Lipschitz continuity of $f'$. 
          In particular,  $x(t{+}1) \leq z(t)$ implies that $\rho|x(t{+}1)-z(t)| < |f'(x(t{+}1))-f'(z(t))|$, which is seen by the following inequality
          \begin{align}
            \rho(z(t)-x(t{+}1))<f'(x(t{+}1))-f'(z(t)) , \label{eq:ADMM-SCA-proof-zzz}
          \end{align}
            which is obtained by combining~\eqref{eq:ADMM-SCA-proof-x-update} and ${-}f'(z(t)){>}g'(x(t))$ and rearranging.
          To see that  $x(t{+}1)<z^{\star}$ we note that 
          \begin{align}
              f'(x){<} {-}(g'(z(t)){+}\rho(x{-}z(t))  ),   \forall x{\in} [z(t),x(t{+}1)[, \label{eq:ADMM-SCA-proof.xxx} \\
             g'(x)<g'(z(t))+\rho(x-z(t)),  ~\forall  x{\in} ]z(t),x(t{+}1)], \label{eq:ADMM-SCA-proof.zzz}
          \end{align}
           where~\eqref{eq:ADMM-SCA-proof.xxx} comes from that $x(t{+}1)$ is the unique solution of~\eqref{eq:ADMM-SCA-proof-x-update} and $f'(z(t)){<}{-}g'(z(t)){-}\rho(z(t){-}z(t))$ and~\eqref{eq:ADMM-SCA-proof.zzz} comes from that $\rho{>}L$ and $g'$ is $L$-Lipschitz continuous.
           Summing~\eqref{eq:ADMM-SCA-proof.xxx} and~\eqref{eq:ADMM-SCA-proof.zzz} and using the continuity of $f'$ and $g'$ shows that $f'(x)+g'(x)<0$ for all $x{\in} [z(t),x(t{+}1)]$ and hence $x(t{+}1){<}z^{\star}$. 
  
     We now show that $z(t{+}1){>}z(t)$ and $z(t){<}z^{\star}$.
     To see that $z(t{+}1){>}z(t)$, we note that $z(t{+}1){\leq} z(t)$ contradicts the $L$-Lipschitz continuity of $g'$. 
     In particular,  $z(t{+}1) {\leq} z(t)$ implies that $\rho|z(t{+}1){-}z(t)| < |g'(x(t{+}1)){-}g'(z(t))|$, which is seen by that if $z(t{+}1) {\leq} z(t)$ then
     \begin{align}
        g'(z(t{+}1))-g'(z(t)) &= \rho(x(t{+}1)-z(t{+}1))  \label{eq:ADMM-SCA-proof-z-largerthan-z1} \\
                                     &> \rho( z(t) - z(t{+}1) )>0 \label{eq:ADMM-SCA-proof-z-largerthan-z2}
     \end{align}
     where~\eqref{eq:ADMM-SCA-proof-z-largerthan-z1} comes by rearranging~\eqref{eq:ADMM-SCA-proof-z-update} and~\eqref{eq:ADMM-SCA-proof-z-largerthan-z2} comes by assuming that   $z(t{+}1) \leq z(t)$ and using that $x(t{+}1)>z(t)$. 
     Hence, we can conclude that $z(t{+}1)>z(t)$.
     To see that  $z(t{+}1)<z^{\star}$ we note that if $z(t{+}1)\leq x(t{+}1)$ then we are done since $x(t{+}1)<z^{\star}$, otherwise we have that 
          \begin{align}
       \hspace{-0.4cm}       f'(z)&{<} {-}(g'(z(t)){+}\rho(x(t{+}1){-}z)  ),   {\forall} z{\in}]x(t{+}1),z(t{+}1)], \label{eq:ADMM-SCA-proof-sec-fff} \\
       \hspace{-0.3cm}       g'(z)&{<}g'(z(t))+\rho(x(t{+}1){-}z),  \forall z{\in} [x(t{+}1),z(t{+}1)[, \label{eq:ADMM-SCA-proof-sec-ggg}
          \end{align}
          where~\eqref{eq:ADMM-SCA-proof-sec-fff} comes from using that $\rho>L$ and $f'$ is $L$-Lipschitz continuous together with the inequalities $z\geq x(t{+}1)$  and $f'(x(t{+}1))<-g'(z(t))$ which follows from~\eqref{eq:ADMM-SCA-proof-x-update}
          and~\eqref{eq:ADMM-SCA-proof-sec-ggg} comes by that $z(t{+}1)$ is the unique solution of~\eqref{eq:ADMM-SCA-proof-z-update} together with that $g'(z(t))<g'(z(t))+\rho(x(t{+}1)-z(t)$ and $z(t)<z(t{+}1)$.
       Summing~\eqref{eq:ADMM-SCA-proof-sec-fff} and~\eqref{eq:ADMM-SCA-proof-sec-ggg} and using the continuity of $f'$ and $g'$ shows that $ f'(z)+ g'(z)<0$ for all $z\in [x(t{+}1),z(t{+}1)]$, implying that $z(t{+}1)<z^{\star}$.

 c) Follows from symmetric arguments as those used for showing b) and is thus omitted.
 \end{proof}
  Informally, Corollary~\ref{Prop:SCA-ADMM} shows that under Assumption~\ref{ADMM-scalar-objs} and $\rho>L$ the ADMM converges to the closest stationary point of $z(0)$ in the direction where $f+g$ is decreasing. 
 For example, when $f(x)=\cos(x)$, $g(z)=\sin(z)$, and $\rho{>}1$ then $\lim_{t \rightarrow \infty}(x(t),z(t),y(t))=(z^{\star},z^{\star},\cos(z^{\star}))$ where $z^{\star}=z(1)$ if $z(1)\in\{2\pi n+ \pi/4 | n\in \mathbb{Z} \}$ and $z^{*}=2n\pi+5\pi/4$ if $z(1)\in ]2n\pi+\pi/4,2(n+1)\pi+\pi/4[$ for $n\in \mathbb{Z}$. 
  If there is no stationary point in the direction where $f+g$ is decreasing then the ADMM diverges to $\pm \infty$, e.g., when $f(x)=g(x)=-x^2$ and $\rho > 2$ then 
  $z^{\star}=0,-\infty,\infty$ for $z(1)=0$, $z(1)<0$, and $z(1)>0$, respectively. 

 The challenge in multidimensional case is that we need to know the direction towards the stationary point.  Such a direction is easily obtained in the monodimensional, as suppose to the multidimensional case.

}




 The next section demonstrates the potential of the proposed ADLM approaches (see Sections~\ref{sec:ADPM} and \ref{sec:ADMM}) in a problem of great practical relevance.

\section{ Application: Cooperative Localization in Wireless Sensor Networks}\label{sec:general applications}

 In this section, we use the ADLM methods to design distributed algorithms for Cooperative Localization (CL)~\cite{Patwari_2005} in wireless sensor networks. 


 Consider an undirected graph $(\mathcal{N},\mathcal{E})$, where $\mathcal{N}= \{1,\cdots, N \}$ is a set of nodes embedded in $\R^2$ 
 \BlueText{and $\mathcal{E}\subseteq \mathcal{N} \times \mathcal{N}$ is a set of edges. }
 Let $\mathcal{N}=\mathcal{S}\cup\mathcal{A}$, where $\mathcal{S}{=}\{1,{\cdots}, S\}$ is the set of sensors with unknown locations and $\mathcal{A}{=}\{ S{+}1,{\cdots}, N\}$ is the set of anchors with known locations. We denote the location of node $n{\in} \mathcal{A}$ by $\vec{a}_n$ and an estimate of the location of node $n{\in} \mathcal{S}$ by~$\vec{z}_n$.

%
%
%
Suppose the measurements of the squared\footnote{Using the square ensures that the objective function of~\eqref{eq:centralized_localization_problem} is continuously differentiable (compare with Assumption~\ref{assumption:ADPM-1}).} distance between two nodes $n,m\in \mathcal{N}$,  denoted by $d_{n,m}^2$, are available if and only if $(n,m)\in\mathcal{E}$. 
 The additive measurement errors are assumed to be independent and Gaussian distributed with zeros mean and variance $\sigma^2$.
 Then the CL problem consists in finding the maximum likelihood estimate of $(\vec{z}_n)_{n\in \mathcal{S}}$ by solving the following problem:
\begin{equation} \label{eq:centralized_localization_problem}
  \begin{aligned}
    & \underset{\vec{z}_1,\cdots, \vec{z}_S\in \R^{2}}{\mbox{minimize}}ß
    &&   {\sum_{n\in \mathcal{S}}} {\Big(}
        {\sum_{ \substack{ m \in \mathcal{S}_n   }}} {\big|} d_{n,m}^2 {-} || \vec{z}_n {-} \vec{z}_m  ||^2 \big|^2   \\
    &&& \hspace{5mm} {+} {2} {\sum_{ \substack{ m \in \mathcal{A}_n}}}\big|  d_{n,m}^2 {-} || \vec{z}_n {-} \vec{a}_m  ||^2 \big|^2  \Big),
  \end{aligned}
\end{equation}
 where $\mathcal{S}_n = \{ m {\in} \mathcal{S} | (n,m) {\in} \mathcal{E} \}$, $\mathcal{A}_n =\{ m {\in} \mathcal{A} {|} (n{,}m) \in  \mathcal{E} \}$ \BlueText{and the coefficient $2$ in front of the second term of the sum comes from that $n\in S$ appears twice in the sum. }
Note that Problem~\eqref{eq:centralized_localization_problem} is NP-hard~\cite{Aspnes_2006}.


To enable distributed implementation (among the nodes) of the proposed ADLM approaches, let us first equivalently reformulate problem~\eqref{eq:centralized_localization_problem} into a general consensus form~\cite[Section~7.2]{Boyd2011_ADMM}. We start by introducing at each node $n\in \mathcal{N}$, a local copy $\vec{x}_n$ of $(\vec{z}_m)_{m\in \bar{\mathcal{S}}_n}$, where $\bar{\mathcal{S}}_n=\mathcal{S}_n\cup \{n \}$. More specifically, we let $\vec{x}_{n} = (\vec{x}_{n,m})_{m\in \bar{\mathcal{S}}_n }$, where $\vec{x}_{n,m}\in\R^2$ denotes the local copy of $\vec{z}_m$ at node $n$. To formally express the consistency between $\vec{x}_n$ and $\vec{z}=(\vec{z}_1,\cdots, \vec{z}_S)$, we introduce the matrix $\vec{E}_n\in \R^{2|\bar{\mathcal{S}}_n|\times 2S}$, which is a $|\bar{\mathcal{S}}_n|\times S$ block matrix of $2\times 2$ blocks. In particular, the $i$-th, $j$-th block of $\vec{E}_n$ is given by $(\vec{E}_n)_{i,j}=\vec{I}_2$, if $\vec{x}_{n,j}$ is the $i$-th block of the vector $\vec{x}_n$ and $(\vec{E}_n)_{i,j}=\vec{0}$ otherwise.
 Then Problem~\eqref{eq:centralized_localization_problem}  is equivalently given by
\begin{equation} \label{eq:distributed_localization_problem}
  \begin{aligned}
    & \underset{\vec{x}, \vec{z}}{\text{minimize}}
    & & \sum_{n\in \mathcal{N}} f_n(\vec{x}_n), \\
    & \text{subject to}
    && \vec{x}_{n}= \vec{E}_n \vec{z}, \mbox{ for all $n\in \mathcal{N}$},
  \end{aligned}
\end{equation}
 where $\vec{x}=(\vec{x}_1 \cdots, \vec{x}_N)\in \R^{\sum_{n\in \mathcal{N}} 2|\bar{\mathcal{S}}_n|}$, $\vec{z}\in \R^{2S}$, and
\be\label{eq:decesion_variables}\nonumber
f_n(\vec{x}_n){=} \hspace{-1mm}\left\{ \begin{array}{ll}
  \hspace{-2mm}\displaystyle\ssum{m\in  \mathcal{S}_n} \big|  d_{n,m}^2 {-} || \vec{x}_{n,n} {-} \vec{x}_{n,m}  ||^2 \big|^2 \\
   \hspace{4mm} {+} \hspace{-2mm}\displaystyle\ssum{m\in \mathcal{A}_n}\hspace{-2mm}\big|  d_{n,m}^2 {-} || \vec{x}_{n,m} {-} \vec{a}_m  ||^2 \big|^2, \hspace{-1mm} & \ \ \text{ if } n\in \mathcal{S}\\
  \hspace{-2mm}\displaystyle\ssum{ m \in \mathcal{S}_n }\big|  d_{n,m}^2 - || \vec{x}_{n,m} - \vec{a}_n  ||^2 \big|^2,  & \ \ \text{ if } n\in \mathcal{A}.
   \end{array} \right.
\ee
%
%

 Problem~\eqref{eq:distributed_localization_problem} fits the form of Problem~\eqref{eq:main_problem_formulation} and proposed ADLM approaches can readily be applied.
 The augmented Lagrangian of problem~\eqref{eq:distributed_localization_problem} can be written~as
\begin{equation*}
     L_{\rho}(\vec{x},\vec{z},\vec{y}) {=}  \hspace{-2mm}\sum_{n\in \mathcal{N}}\hspace{-1mm} f_n(\vec{x}_n){+} \vec{y}_n\tran (\vec{x}_n{-}\vec{E}_n\vec{z}) {+} \frac{\rho}{2}|| \vec{x}_n{-}\vec{E}_n\vec{z}||^2,
\end{equation*}
 where $\vec{y}=(\vec{y}_1,\cdots, \vec{y}_n)$ is the Lagrangian multiplier.
 Note that the variables $\vec{x}$ and $\vec{y}$ are separable among
 $n\in\mathcal{N}$.
 The resulting distributed-ADLM is as follows.

{\tableANDalgSize

    \noindent\rule{\linewidth}{0.3mm}
\\
\emph{Algorithm 3: \ \textsc{Distributed Alternating Direction Lagrangian Method (\small{D-ADLM}) }}

\vspace{-0.2cm}
\noindent\rule{\linewidth}{0.3mm}
\begin{enumerate}
   \item \label{Alg:DADLM-initialization-step}  \textbf{Initialization:} Set $t=0$ and put initial values to $\vec{z}(t)$, $\vec{y}(t)$, and $\rho(t)$. 
   \item \label{Alg:DADLM-x-step} \textbf{Subproblem:}  Each node $n\in \mathcal{N}$ solves
                                                      \begin{equation} \label{eq:ADLM_for_CL_in_WSN_x_update}
                                                              \vec{x}_n(t+1)=\underset{\vec{x}_n \in \R^{|\mathcal{S}_n|}}{\text{argmin}}~  L_{\rho(t)}(\vec{x}_n,\vec{z}(t),\vec{y}(t))
                                                      \end{equation}
   \item \label{Alg:DADLM-z-step} \textbf{Communication/Averaging:}
                                                     $\vec{z}(t+1)$ is given by $\mbox{argmin}_{\vec{z}}~  L_{\rho(t)}(\vec{x}_n(t{+}1),\vec{z},\vec{y}(t))$, i.e.,
                                                    \begin{multline}\label{eq:ADLM_for_CL_in_WSN_z_update}
                                                      \hspace{-4mm}       \vec{z}_n(t {+} 1) {=} \frac{1}{|\mathcal{S}_n|}  \sum_{i\in \mathcal{S}_n }  \vec{E}_{i,n}\tran \left( \vec{x}_i(t{+}1){+} \frac{\vec{y}_i(t)}{\rho(t)}\right),
                                                    \end{multline}
      for $n\in \mathcal{S}$, where $\vec{E}_{i,n}$ is the column $n$ of the block matrix $\vec{E}_i$. 

    \item   \label{Alg:DADLM-y-step} \textbf{Local parameter update:} Each node $n\in \mathcal{N}$ updates its local parameters $\rho(t)$ and $\vec{y}(t)$ accordingly.
   \item \label{Alg:DADLM-stop} \textbf{Stopping criterion:} If stopping criterion is met terminate, otherwise set $t=t+1$ and go to step~2.
\end{enumerate}
\vspace{-3mm}
\rule{\linewidth}{0.3mm}

}

\noindent Note that the D-ADLM can be carried out either as an ADPM or as an ADMM by performing $\rho(t)$ and $\vec{y}(t)$ updates at step~4 accordingly (compare with step 4 of ADPM and ADMM algorithms).
\BlueText{
 In particular, in ADPM, all the nodes know the value of $\rho(t)$ for each $t$ and the nodes can update $\vec{y}_n(t)$, for all $n\in \mathcal{N}$, as they  wish, as long as the sequence $\vec{y}_n(t)$ is bounded.
 In ADMM, all the nodes $n$ know the value of $\rho$ and update $y_n$ according to
 \begin{equation} \label{secNum:dualUp}
   \vec{y}_n(t{+}1)=\vec{y}_n(t)+ \rho(t)(\vec{x}_n(t{+}1)-\vec{E}_n\vec{z}(t{+}1)).
 \end{equation}
}

As indicated in the first step, the initial setting of the algorithm should be agreed on among the nodes. Other steps can be carried out in a distributed manner with local message exchanges.
Note that~\eqref{eq:ADLM_for_CL_in_WSN_z_update} is simply the average of the local copies of $\vec{z}_n$ and the corresponding dual variables [scaled by $\rho(t)$], which can be performed by employing standard gossiping algorithms, e.g.,~\cite{Boyd_2006}. Moreover, the last step requires a mechanism to terminate the algorithm.  A natural stopping criterion is to fix the number of iterations, which requires no coordination among the nodes except at the beginning. In order to control the accuracy level $\epsilon$ of the coupling constraints,  one can, for example, terminate the algorithm when ${\max_{n\in\mathcal{N}}}|| \vec{x}_n(t){-}\vec{E}_n\vec{z}(t)|| {<} \epsilon$.  This, can be accomplished with an additional coordination among the~nodes.

\BlueText{
 We compare D-ADLM with the following distributed gradient descent algorithm.

{\tableANDalgSize
    \noindent\rule{\linewidth}{0.3mm}
\\
\emph{Algorithm 4: \ \textsc{Distributed Gradient Descent (\small{D-GD}) }}

\vspace{-0.2cm}
\noindent\rule{\linewidth}{0.3mm}
\begin{enumerate}
   \item \label{Alg:DADLM-initialization-step}  \textbf{Initialization:} Set $t{=}0$ and initialize $\rho(t)$, $\vec{z}(t)$, and $\bar{x}_n(t)=\vec{E}_n \vec{z}(t)$ for all $n\in \mathcal{N}$.
   \item \label{Alg:DADLM-x-step} \textbf{Subproblem:}  Each node $n\in \mathcal{N}$ solves
                                                      \begin{equation} \label{eq:DGD_for_CL_in_WSN_x_update}
                                                              \vec{x}_n(t+1)= \bec{x}_n(t) - \frac{1}{\rho(t)} \nabla f_n (\bec{x}_n(t) )
                                                      \end{equation}
   \item \label{Alg:DADLM-z-step} \textbf{Communication/Averaging:}
            Each senor $n\in \mathcal{S}$ finds the average estimation of its localization by communicating with neighbors:
                                                    \begin{multline}\label{eq:DGD_for_CL_in_WSN_z_update}
                                                          \vec{z}_n(t {+} 1) {=} \frac{1}{|\mathcal{S}_n|}  \sum_{i\in \mathcal{S}_n }  \vec{E}_{i,n}\tran  \vec{x}_i(t{+}1),
                                                    \end{multline}
       here $\vec{E}_{i,n}$ is the column $n$ of the block matrix $\vec{E}_i$.
      Set $\bec{x}_n(t{+}1)=\vec{E}_n \vec{z}(t{+}1)$, i.e., the average of the components pertaining to $n\in \mathcal{S}$.

    \item   \label{Alg:DADLM-y-step} \textbf{Local parameter update:} Each node $n{\in} \mathcal{N}$ updates  $\rho(t)$.
   \item \label{Alg:DADLM-stop} \textbf{Stopping criterion:} If stopping criterion is met terminate, otherwise set $t=t+1$ and go to step~2.
\end{enumerate}
\vspace{-3mm}
\rule{\linewidth}{0.3mm}
}

 Note that D-GD performs almost the same steps as D-ADLM. 
 The main difference is in step 2): \eqref{eq:ADLM_for_CL_in_WSN_x_update} in ADLM is a solution to an optimization problem while \eqref{eq:DGD_for_CL_in_WSN_x_update} in D-GD is a gradient descent step. 
  In particular, the required communication is the same for both algorithms.
 Therefore, D-GD provides a fair comparison to the D-ADLM.

  Let us next test the D-ADLM on a CL problem. 

}

\BlueText{
\subsection{Numerical Results}

\begin{figure}[t]
\centering
\subfigure[Residuals]
{\includegraphics[width=0.23\textwidth]{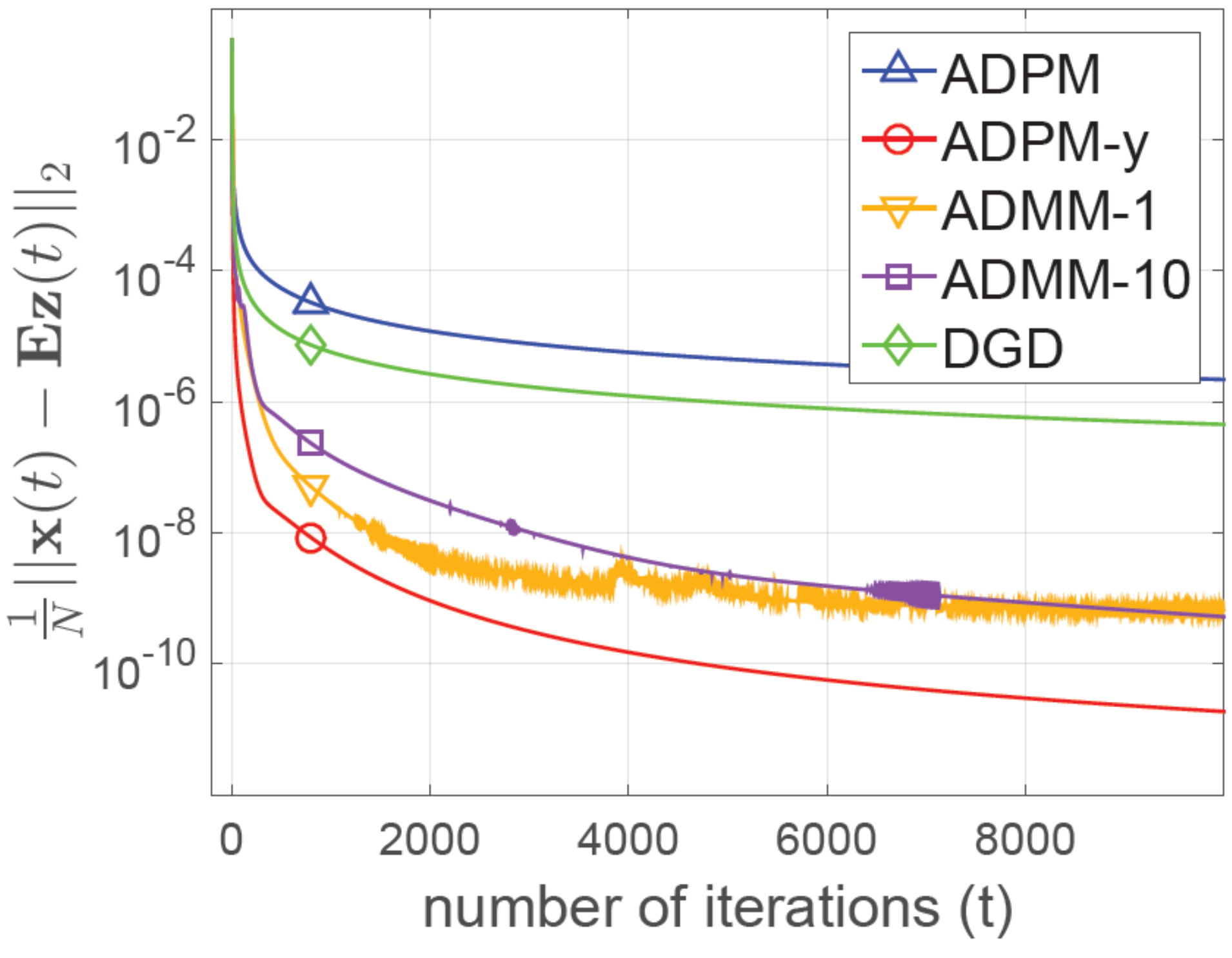}
\label{fig:Localization_res}}
\subfigure[Gradient of the objective function]
{\includegraphics[width=0.23\textwidth]{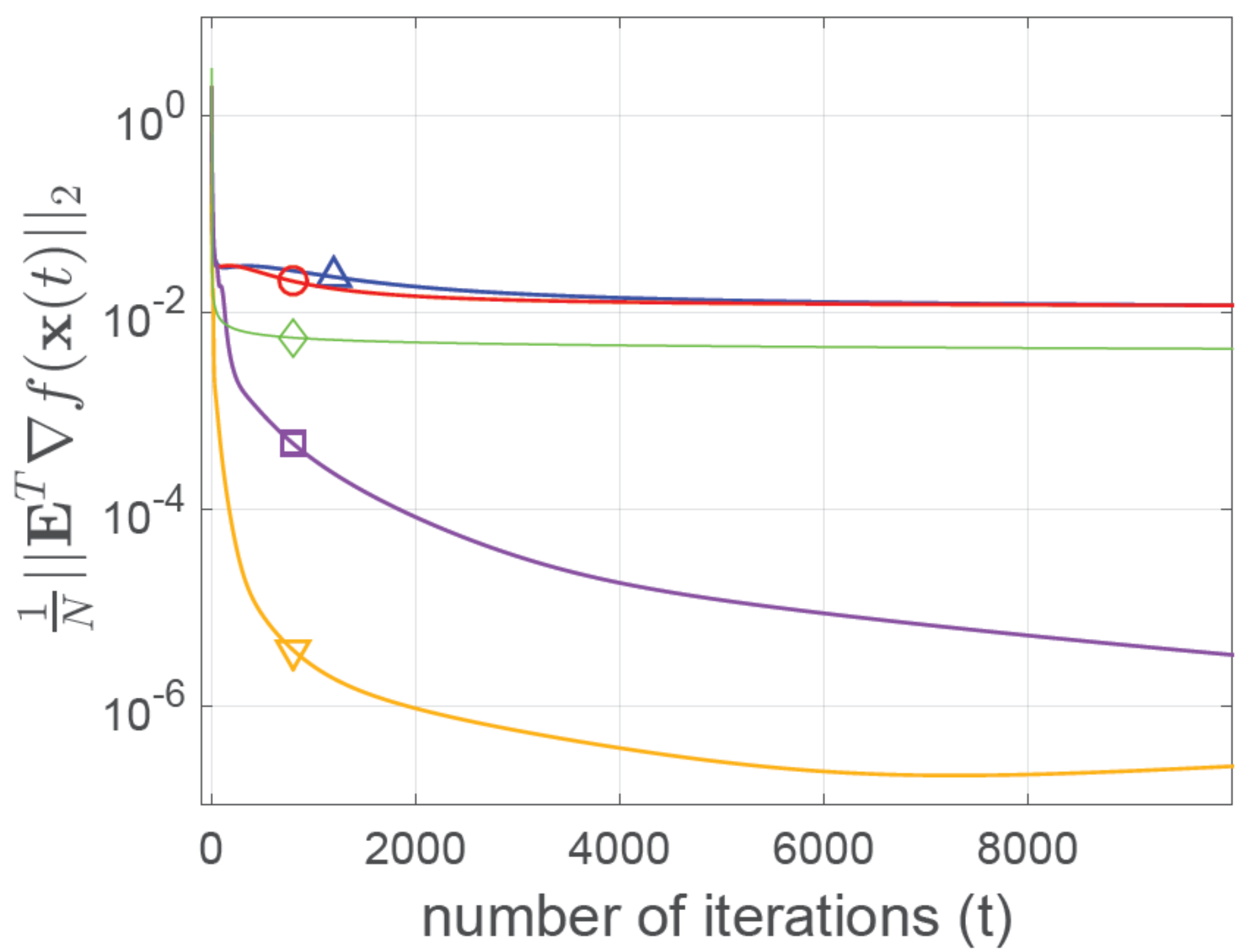}
\label{fig:Localization_grad}}\vspace{-2mm}
\subfigure[Dual variables]
{\includegraphics[width=0.23\textwidth]{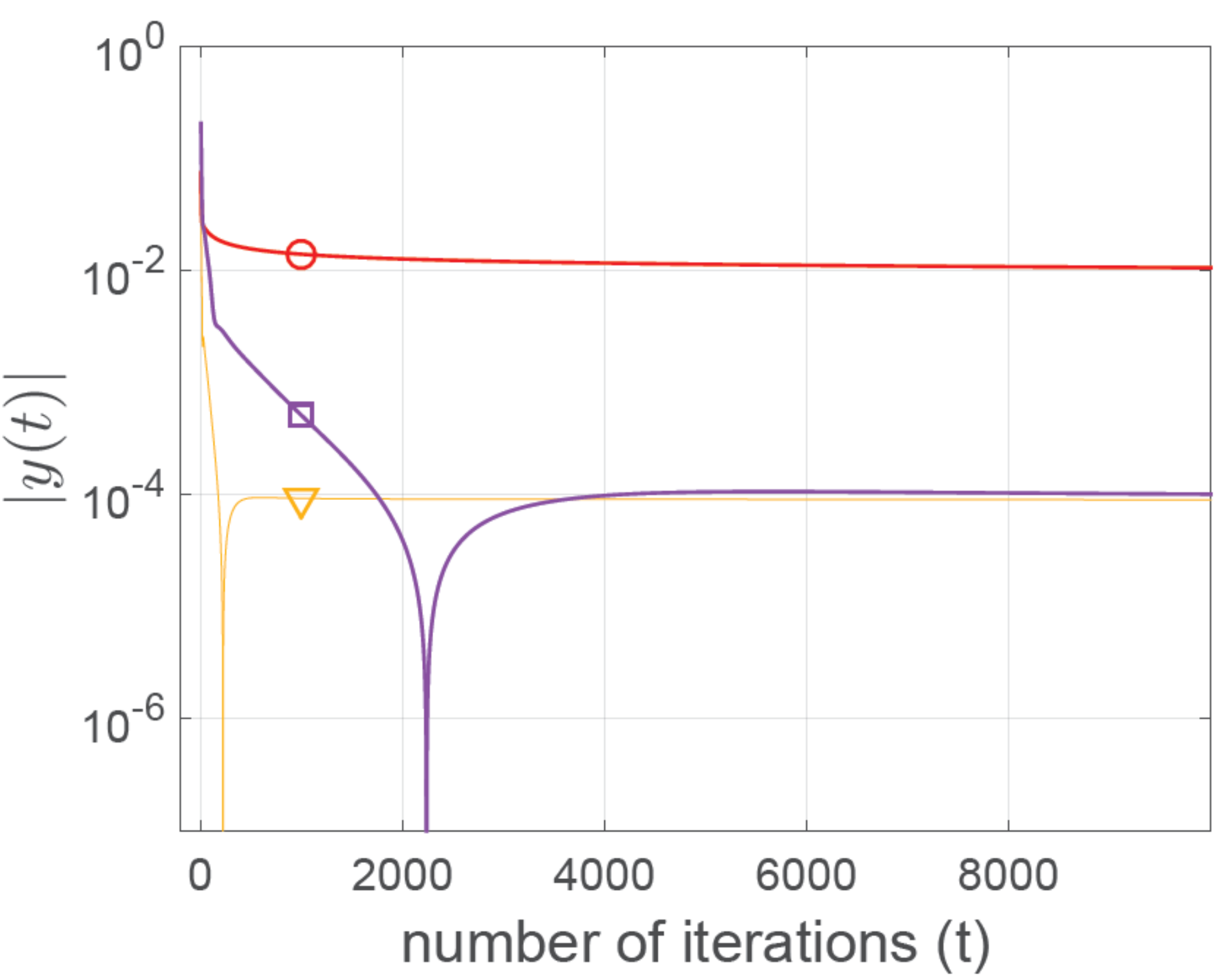}
\label{fig:Localization_lam}}
\subfigure[Objective function]
{\includegraphics[width=0.23\textwidth]{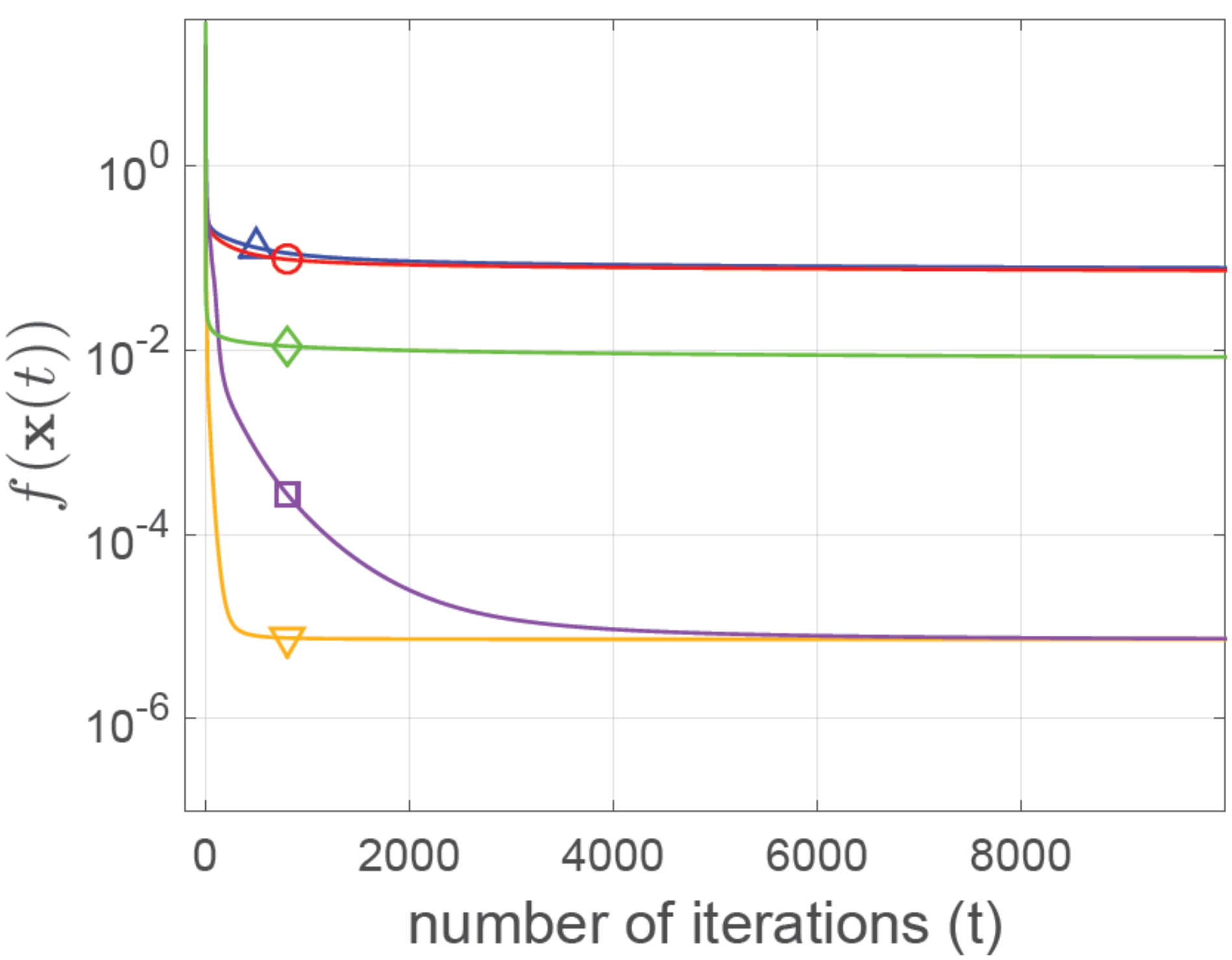}
\label{fig:Localization_obj}}\vspace{-2mm}
\caption{
 The results of running all 5 algorithms on the test network.  
}
\label{fig:Localization_results}
\vspace{-4mm}
\end{figure}

\begin{figure}[t]
\centering{
\includegraphics[width=0.5\textwidth]{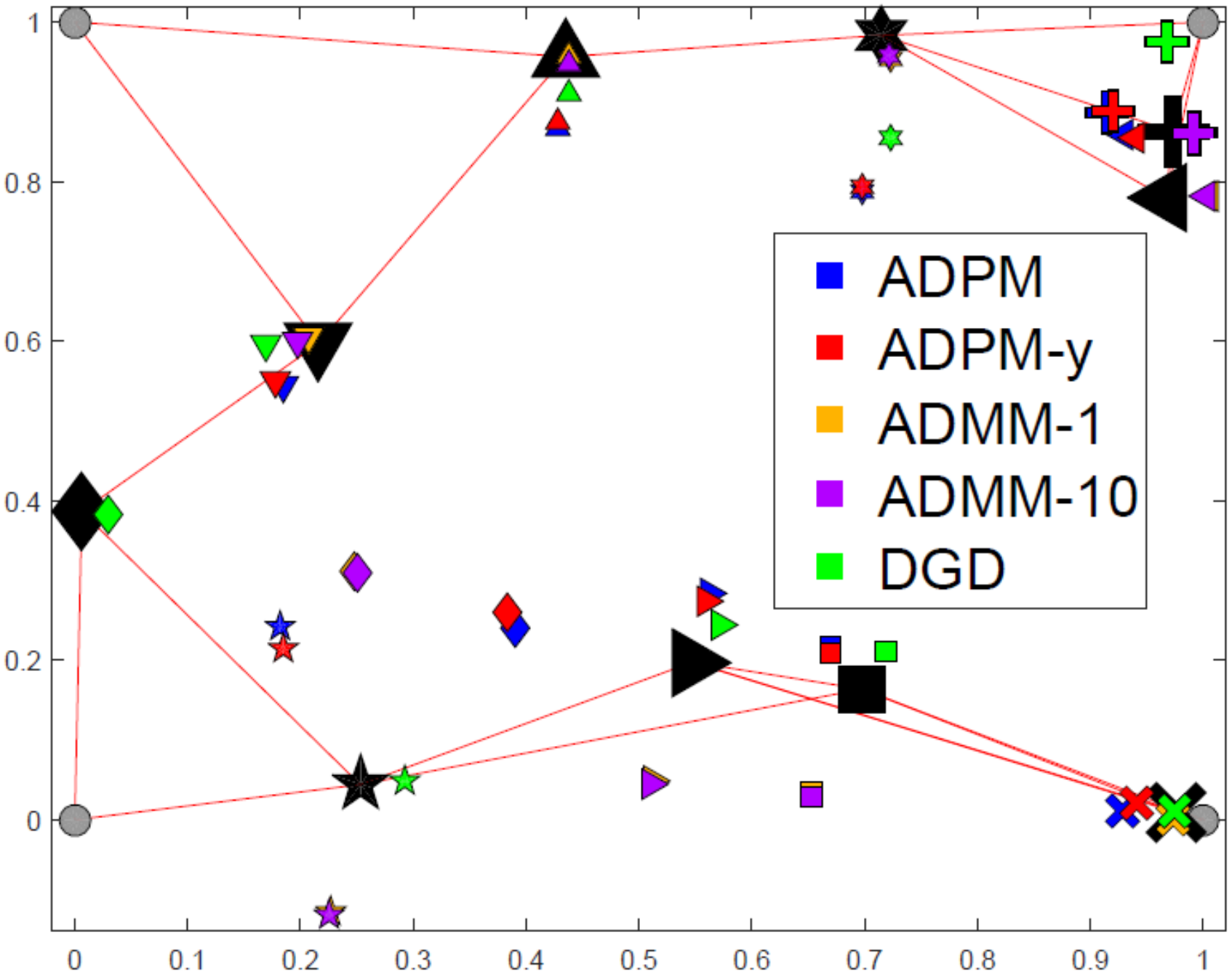} }
\caption{
 The position estimate each algorithm converges to. 
}
\label{fig:Localization_results2}
\vspace{-4mm}

\end{figure}

 We consider a network with $S=10$, $A=4$.
 The 4 anchors are located at $(0,0),$ $(0,1),$ $(1,0),$ and $(1,1)$.
  The senors' are positioned at uniform random in $[0,1]{\times} [0,1]$.
  There is an edge between two nodes $n,m {\in} \mathcal{N}$ if and only if the Euclidean distance between those is less than $0.5$.
  We have $\sigma^2{=} 0.05D$, where $D$ is the average squared distance between distinct nodes $(n,m){\in} \mathcal{E}$.
  We consider the following algorithm settings:

\begin{center} \tableANDalgSize
\begin{tabular}{ |c|c|c|c| }
  \hline
   name & type & dual update & $\rho$ \\
  \hline
  \texttt{ADPM}       & ADPM    & None      & $\rho(t)=t$\\
  \texttt{ADPM-y}    & ADPM    &  \eqref{secNum:dualUp} & $\rho(t)=t$\\
 \texttt{ADMM-1}    & ADMM    & \eqref{secNum:dualUp}  & $\rho =1$   \\
 \texttt{ADMM-10}  & ADMM    & \eqref{secNum:dualUp}  &  $\rho =10$   \\
 \texttt{DGD}          & D-GD       & None     & $\rho(t)=t$\\
\hline
\end{tabular}
\end{center}
 where the first column identifies each setting,
  the second column indicates the algorithm used, 
  the third column indicates whether the dual variable~update 
 is used or if no dual variable update is used, i.e., $\vec{y}(t){=}\vec{0}$,
 the forth column indicates the penalty/steps size used.
 We initialize the algorithms as $\vec{z}_n(0){=}(0.5,0.5)$ for all $n{\in} \mathcal{S}$.
 When the dual variable is updated we initialize it as~$\vec{y}(0){=}\vec{0}$.



 Fig.~\ref{fig:Localization_results} depicts the results, where we have compactly written $\vec{x}=(\vec{x}_1,\cdots,\vec{x}_n)$ and $\vec{E}=(\vec{E}_1,\cdots,\vec{E}_n)$. 
 Figs~\ref{fig:Localization_res} and~\ref{fig:Localization_grad} depict scaled versions of $||\vec{x}(t)-\vec{E}\vec{z}(t)||$, the network consensus, and $||\vec{E}\tran \nabla f(\vec{x}(t))||$, the gradient of the objective function, respectively,  as a function of iterations $t$.  
 Together $||\vec{x}(t)-\vec{E}\vec{z}(t)||$ and $||\vec{E}\tran \nabla f(\vec{x}(t))||$ comprise the FON conditions of Problem~\eqref{eq:distributed_localization_problem}, i.e., when both quantities converge to zero the FON conditions is asymptotically reached.
  Both Figs~\ref{fig:Localization_res} and~\ref{fig:Localization_grad}  demonstrate a decreasing trend for all algorithms.
  In Fig.~\ref{fig:Localization_grad}, \texttt{DGD} and \texttt{ADPM} have noticeably slower decay rate than \texttt{ADPM-y},  \texttt{ADMM-1}, and \texttt{ADMM-10}.
  In Fig.~\ref{fig:Localization_grad}, \texttt{DGD}, \texttt{ADPM}, and \texttt{ADPM-y} have noticeably slower decay rate than \texttt{ADMM-1} and \texttt{ADMM-10}.
 Therefore, the results suggest that it can be beneficial to use the update~\eqref{secNum:dualUp}.

   Fig.~\ref{fig:Localization_lam} depicts an example of a dual variable for each of the algorithms where the  the update~\eqref{secNum:dualUp} is used.
  Similar results were observed for the other dual variables. 
  The figure shows that the dual variables converge, implying that \texttt{ADMM-1} and \texttt{ADMM-10} converge based on Proposition~\ref{prop:ADMM-main-prop}.

  Fig.~\ref{fig:Localization_res} depicts the objective value at each iteration.
  The algorithms achieve different objective values, which is not surprising since the objective is function nonconvex with multiple local minima. 
  Fig~\ref{fig:Localization_results2} depicts the resulting location estimations for each algorithm, i.e., the estimation at the final iteration.
  Note that the orange diamond and five-pointed star lay under their purple counter parts and are therefore not visible in the figure.
  Despite the nonconvexities, all the algorithms converge to a good estimations close to the true locations of the nodes. 
  The \texttt{DGD} achieves a visually better estimation of the diamond and the five-pointed star in Fig~\ref{fig:Localization_results2} than the other algorithms.
  Nevertheless,  \texttt{ADMM-1} and \texttt{ADMM-10}  achieve much better objective function values. 

 \begin{remark}
    The gradients of $f_n$ for $n\in \mathcal{N}$ are unbounded, but still Assumption~\eqref{assumption:ADPM-unCon-2}.b holds, which  ensures that the sequence $(\vec{x}(t),\vec{z}(t))$ of \texttt{ADPM} is bounded, see Lemma~\ref{sec3ADPM:Lemma1}.
  Similar results can be derived for \texttt{ADPM-y}, \texttt{ADMM-1}, and \texttt{ADMM-10} as long as the dual variables $\vec{y}(t)$ are bounded.
 On the other hand, from our numerical experiences, \texttt{DGD} turned out to be unstable for many initialisations where it reached floating point infinity in only few iterations.

 \end{remark}

}

  \section{Conclusions}
    We investigated the convergence behaviour of scalable variants of two standard nonconvex optimization methods:
 a novel method we call Alternating Direction Penalty Method and the well known Alternating Direction Method of Multipliers, variants of the Quadratic Penalty Method and the Method of Multipliers, respectively.
 Our theoretical results showed the ADPM asymptotically reaches primal feasibility under assumptions that hold widely in practice and provided sufficient conditions for when ADPM asymptotically reaches the first order necessary conditions for optimality. 
  Furthermore, we provided sufficient conditions for the asymptotic convergence of ADMM to the first order necessary condition for local optimality and provided a class of problems where thous conditions hold.
 Finally, we demonstrated how the methods can be used to design distributed algorithms for \emph{nonconvex} cooperative localization in wireless sensor networks. 

\bibliography{refs}
\bibliographystyle{IEEEtran}

\end{document}